%% file: NSR2+.tex
\documentclass[smallextended,referee,envcountsect]{svjour3}
\usepackage{amsfonts,amssymb,amsmath}
\usepackage{cite}
\usepackage{color}
\usepackage[all]{xy}
\usepackage[colorlinks,urlcolor=red]{hyperref}
\smartqed
\input{macro}
\def\sr{{}^sr}

\def\smartqedtr{\def\qedtr{\ifmmode\triangle\else{\unskip\nobreak\hfil
\penalty50\hskip1em\null\nobreak\hfil$\triangle$
\parfillskip=0pt\finalhyphendemerits=0\endgraf}\fi}}
\smartqedtr  % flush right qed marks, e.g. at end of proof
\lccode`\-=`\-

\renewcommand {\theenumi} {\roman{enumi}}

\begin{document}

\title{Nonlinear Metric Subregularity}
%\thanks{The research was supported by the Australian Research Council, project DP110102011.}}

\author{Alexander Y. Kruger}
\institute{A. Y.~Kruger \at
Centre for Informatics and Applied Optimization, Faculty of Science and Technology, Federation University Australia, POB 663, Ballarat, Vic, 3350, Australia\\
\email{a.kruger@federation.edu.au} }

%\dedication{Dedicated to the memory of Vladimir Demyanov}

\date{Received: date / Accepted: date}

%\journalname{J Optim Theory Appl}
\maketitle

% ----------------------------------------------------------------
\begin{abstract}
In this article, we investigate nonlinear metric subregularity properties of set-valued mappings between general metric or Banach spaces.
We demonstrate that these properties can be treated in the framework of the theory of (linear) error bounds for extended real-valued functions of two variables developed in \textsf{A.~Y.~Kruger, Error bounds and metric subregularity, Optimization \textbf{64}, 1 (2015) 49--79}.
Several primal and dual space local quantitative and qualitative criteria of nonlinear metric subregularity are formulated.
The relationships between the criteria are established and illustrated.
\end{abstract}


%\thispagestyle{empty}

% ----------------------------------------------------------------
%\tableofcontents
\section{Introduction}

The linear \emph{metric subregularity} property of set-valued mappings (cf., e.g., \cite{RocWet98,Mor06.1,DonRoc14,Pen13})
%; see definition \eqref{MR0++} below)
and closely related to it \emph{calmness} property play an important role in both theory and applications.
The amount of publications devoted to (mostly sufficient) primal and dual criteria of linear metric subregularity is huge.
The interested reader is referred to \cite{HenOut01,HenJou02,HenJouOut02,DonRoc04, ZheNg07,IofOut08,DonRoc14, Lev09,Pen10,ZheNg10,ZheOuy11,ZheNg12, ApeDurStr13,Kru15} and the references therein.

In many important applications like e.g. analysis of sensitivity and controllability in optimization and control (cf. Ioffe \cite{Iof13}), the standard linear metric (sub)regularity property is not satisfied, and more subtle nonlinear, mostly H\"older type estimates come into play.
%(see definition \eqref{02} below).
The H\"older version of the more robust \emph{metric regularity} property
and even more general nonlinear regularity models have been studied since 1980s; cf. \cite{Fra87,BorZhu88,Fra89,Pen89,Jou91,YenYaoKie08, Iof00_, FraQui12,Ude12.2,
Iof13, Kru15.2}.
The history of the nonlinear/H\"older metric subregularity property seems to be significantly shorter with most work done in the last few years, cf. \cite{Kum09,GayGeoJea11,KlaKruKum12,LiMor12, MorOuy,NgaiTronThe3,NgaiTronThe4,NgaiTin15}, although some studies of such properties can be found in earlier publications, cf. e.g. Klatte \cite{Kla94} and Cornejo, Jourani \& Z{\u{a}}linescu \cite{CorJouZal97}.
%where a relation between the so called \emph{$\psi$-conditioning} of a function and the \emph{upper $\varphi$-Lipschitz property} of the inverse of the subdifferential was established in the convex and nonconvex settings.
%Note the only attempt so far to consider the case $q>1$ in \cite{MorOuy}.
To the best of our knowledge, general nonlinear subregularity models have not been studied so far.

There exists strong similarity between the definitions and criteria of linear and nonlinear metric subregularity of set-valued mappings and the well developed theory of error bounds (cf. \cite{Pang97,Aze03,AzeCor04,NgaiThe04,CorMot08,IofOut08, NgaiThe08, FabHenKruOut10,FabHenKruOut12}) of extended real-valued functions.
However, there is an obstacle, which prevents direct application of this theory to deducing criteria of metric subregularity, namely, the function involved in the definition of the metric subregularity property, in general, fails to be \lsc.
Nevertheless, many authors use error bound type arguments when proving metric subregularity criteria.
For that, they define auxiliary functions, which possess the lower semicontinuity property.
The details are usually hidden in the proofs.

Such an approach has been formalized and made explicit in \cite{Kru15} where the theory of local (linear) error bounds has been extended to functions on the product of metric spaces and applied to deducing linear metric subregularity criteria for set-valued mappings.
This extended theory of linear error bounds is applicable also to nonlinear subregularity models.
This has been demonstrated in \cite{Kru15.2} where H\"older metric subregularity has been investigated.
The current article targets several general settings of nonlinear metric subregularity, namely $f$-subregularity, $g$-subregularity, and $\varphi$-subregularity with each next regularity type being a special case of the previous one while H\"older metric subregularity is a special case of metric $\varphi$-subregularity.
This hierarchy of regularity properties translates naturally into the corresponding hierarchy of regularity criteria, illustrating clearly the relationship between the assumptions on the set-valued mapping, the regularity property under investigation and the resulting regularity criteria.

Following the standard trend initiated by Ioffe \cite{Iof00_} and Az\'e and Corvellec \cite{AzeCor04}, criteria for error bounds and metric subregularity of set-valued mappings in metric spaces are formulated in terms of (strong) \emph{slopes} \cite{DMT}.
Following \cite{Kru15,Kru15.2}, to simplify the statements in metric and also Banach/Asplund spaces, several other kinds of primal and dual space slopes for real-valued functions and set-valued mappings are discussed in this article and the relationships between them are formulated.
These relationships lead to a simple hierarchy of the error bound and metric subregularity criteria.

Recall that a Banach space is \emph{Asplund} iff the dual of each its separable subspace is separable; see, e.g., \cite{Mor06.1,BorZhu05} for discussions and characterizations of Asplund spaces.
Note that any \emph{Fr\'echet smooth} space, i.e. a Banach space, which admits an equivalent norm Fr\'echet differentiable at all nonzero points, is
Asplund.
Given a Fr\'echet smooth space, we will always assume that it is endowed with such a norm.

Some statements in the article look rather long because each of them contains an almost complete list of criteria applicable in the situation under consideration.
The reader is not expected to read through the whole list.
Instead, they can select a particular criterion or a group of criteria corresponding to the setting of interest to them (e.g., in metric or Banach/Asplund/smooth spaces, in the convex case, etc.)

The structure of the article is as follows.
The next section provides some preliminary definitions and facts, which are used throughout the article.
In Section~\ref{S1}, we present a survey of error bound criteria for a special family of extended-real-valued functions on the product of metric or Banach/Asplund spaces from \cite{Kru15,Kru15.2}.
The criteria are formulated in terms of several kinds of primal and subdifferential slopes.
The relationships between the slopes are presented.
Section~\ref{S3} is devoted to nonlinear metric subregularity of set-valued mappings with the main emphasis on metric $g$-sub\-regularity.
We demonstrate how the definitions of slopes and error bound criteria from Section~\ref{S1} translate into the corresponding definitions and criteria for metric $g$-subregularity.
Some new relationships between the slopes are established and, in finite dimensions, new objects -- limiting $g$-coderivatives -- are introduced and then used in dual space criteria of metric $g$-sub\-regularity.
In Section~\ref{S50}, we study a particular case of metric $g$-subregularity called metric $\varphi$-subregularity and using sharper tools (slopes) derive more specific regularity criteria.
The last section contains concluding remarks.

\section{Preliminaries}

Recall that
a set-valued mapping
$F:X\rightrightarrows Y$ is a mapping, which
assigns to every $x\in X$ a subset (possibly empty) $F(x)$ of $Y$.
We use the notation
$$\gph F:=\{(x,y)\in X\times Y\mid
y\in F(x)\}$$
for the graph of $F$ and $F\iv : Y\rightrightarrows
X$ for the inverse of $F$.
This inverse (which always exists with possibly empty values) is defined by
$$F\iv(y) :=\{x\in X|\, y\in F(x)\},\quad y\in Y,$$
and satisfies
$$(x,y)\in\gph F \quad\Leftrightarrow\quad (y,x)\in\gph F\iv.$$
If $X$ and $Y$ are linear spaces, we say that $F$ is convex iff $\gph F$ is a convex subset of $X\times Y$.

A set-valued mapping $F:X\rightrightarrows Y$ between metric spaces is called (locally) \emph{metrically subregular} at a point $(\bx,\by)\in\gph F$ with constant $\tau>0$ iff there exists a neighbourhood $U$ of $\bx$,
such that
\begin{equation}\label{MR0++}
\tau d(x,F^{-1}(\by))\le d(\by,F(x)) \quad \mbox{ for all } x \in U.
\end{equation}

This property represents a weaker version of the more robust \emph{metric regularity} property, which corresponds to replacing $\by$ in the above inequality by an arbitrary (not fixed!) $y$ in a neighbourhood of $\by$.

If instead of \eqref{MR0++} one uses the following more general condition:
\begin{equation}\label{02}
\tau d(x,F^{-1}(\by))\le (d(\by,F(x)))^q\quad \mbox{for all } x\in U,
\end{equation}
where $q\in(0,1]$, then the corresponding property is usually referred to as \emph{H\"older metric subregularity} of order $q$ at $(\bx,\by)$ with constant $\tau$.
The case $q=1$ corresponds to standard (linear) metric subregularity.
If $q_1<q_2\le1$, then H\"older metric subregularity of order $q_1$ is in general weaker than that of order $q_2$.

If fixed $\by$ in the above inequality is replaced by an arbitrary $y$ and the inequality is required to hold uniformly over all $y$ near $\by$, then we arrive at the definition of \emph{H\"older metric regularity} of order $q$.

One can easily see that H\"older metric subregularity property \eqref{02} is equivalent to the local error bound property of the extended real-valued function $x\mapsto(d(\by,F(x)))^q$ at $\bx$ (with the same constant).
So one might want to apply to this model the well developed theory of error bounds.
However, most of the error bound criteria are formulated for \lsc\ functions, while the function $x\mapsto(d(\by,F(x)))^q$ can fail to be \lsc\ even when $\gph F$ is closed.

Another helpful observation is that property \eqref{02} can be rewritten equivalently as
\begin{equation*}
\tau d(x,F^{-1}(\by))\le (d(\by,y))^q\quad \mbox{for all } x\in U,\; y\in F(x),
\end{equation*}
or
\begin{equation}\label{2f}
\tau d(x,F^{-1}(\by))\le f(x,y)\quad \mbox{for all } x\in U,\; y\in Y,
\end{equation}
where
\begin{equation}\label{2f2}
f(x,y):=
\begin{cases}
(d(y,\by))^q& \text{if } (x,y)\in\gph F,\\
+\infty & \text{otherwise}.
\end{cases}
\end{equation}

One can also consider property \eqref{2f} with $f:X\times Y\to\R_+\cup\{+\infty\}$ being a more general than \eqref{2f2} nonlinear function.
This property,  which we refer to as \emph{metric $f$-sub\-regularity}, is the main object of our study in this article.
The assumptions on function $f$, which are going to be specified in the next two sections allow us to treat property \eqref{2f} in the framework of the extended theory of error bounds of functions of two variables developed in \cite{Kru15} and used there and in \cite{Kru15.2} for characterizing linear and H\"older metric subregularity, respectively.

Two special cases of metric $f$-subregularity are of special interest: when
$$f(x,y)=g(y)+ i_{\gph F}(x,y),\quad x\in X,\; y\in Y,$$
where $g:Y\to\R_+$ and $i_{\gph F}$ is the \emph{indicator function} of $\gph F$ ($i_{\gph F}(x,y)=0$ if $(x,y)\in\gph F$ and $i_{\gph F}(x,y)=\infty$ otherwise) and when
$$g(y)=\varphi(d(y,\by)),\quad y\in Y,$$
where $\varphi:\R_+\to\R_+$.
We refer to these two properties as \emph{metric $g$-sub\-regularity} and \emph{metric $\varphi$-sub\-regularity}, respectively.
The particular assumptions on $g$ and $\varphi$ are discussed when the properties are defined.

Our basic notation is standard, see \cite{Mor06.1,RocWet98,DonRoc14,Pen13}.
Depending on the context, $X$ and $Y$ are either metric or normed
spaces.
Metrics in all spaces are denoted by the same symbol
$d(\cdot,\cdot)$;
$d(x,A):=\inf_{a\in{A}}d(x,a)$ is the point-to-set distance from
$x$ to $A$.
$B_\de(x)$ denotes the closed ball with radius $\de\ge0$ and centre $x$.
If not specified otherwise, the
product of metric/normed spaces is assumed equipped with the distance/norm given by the maximum of the distances/norms.

If $X$ and $Y$ are normed spaces, their topological duals are denoted $X^*$ and $Y^*$, respectively, while $\langle\cdot,\cdot\rangle$ denotes the bilinear form defining the pairing between the spaces.
The closed unit balls in a normed space and its dual are denoted by $\B$ and $\B^*$, respectively, while $\Sp$ and $\Sp^*$ stand for the unit spheres.

We say that a subset $\Omega$ of a metric space is locally closed near $\bar{x}\in\Omega$ iff $\Omega\cap{U}$ is closed for some closed neighbourhood $U$ of $\bar{x}$.
Given an $\al\in\R_\infty:=\R\cup\{+\infty\}$, $\alpha_+$ denotes its ``positive'' part:
$\alpha_+:=\max\{\alpha,0\}$.

If $X$ is a normed linear space, $f:X\to\Real_\infty$, $x\in{X}$, and $f(x)<\infty$, then
\begin{gather}
\partial{f}(x) := \left\{x^\ast\in X^\ast\bigl|\bigr.\;
\liminf\limits_{u\to x,\,u\ne x}
\frac{f(u)-f(x)-\langle{x}^\ast,u-x\rangle}
{\norm{u-x}} \ge 0 \right\}\label{Frsd}
\end{gather}
is the \emph{Fr\'echet subdifferential} of $f$ at $x$.
Similarly, if $x\in\Omega\subset X$, then
\begin{gather}
N_{\Omega}(x) := \left\{x^\ast\in X^\ast\bigl|\bigr.\
\limsup_{u\to x,\,u\in\Omega\setminus\{x\}} \frac {\langle x^\ast,u-x\rangle}
{\|u-x\|} \le 0 \right\}\label{Fr}
\end{gather}
is the \emph{Fr\'echet normal cone} to $\Omega$ at $x$.
In the convex case, sets \eqref{Frsd} and \eqref{Fr} reduce to the subdifferential and normal cone in the sense of convex analysis, respectively.
If $f(x)=\infty$ or $x\notin\Omega$, we set, respectively, $\partial{f}(x)=\emptyset$ or $N_{\Omega}(x)=\emptyset$.
Observe that definitions \eqref{Frsd} and \eqref{Fr} are invariant on the renorming of the space (replacing the norm by an equivalent one).

If $F:X\rightrightarrows Y$ is a set-valued mapping between normed linear spaces and $(x,y)\in\gph F$, then
 \[
D^{*}F(x,y)(y^{*}):=\set{x^{*}\in X^*\mid (x^{*},-y^{*})\in N_{\gph F}(x,y)},\quad y^{*}\in X^*
 \]
is the {\em Fr\'echet coderivative} of $F$ at $(x,y)$.

The proofs of the main statements rely on several kinds of \emph{subdifferential sum rules}.
Below we provide these results for completeness.

\begin{lemma}[Subdifferential sum rules] \label{l02}
Suppose $X$ is a normed linear space, $f_1,f_2:X\to\R_\infty$, and $\bx\in\dom f_1\cap\dom f_2$.

{\rm (i) \bf Fuzzy sum rule}. Suppose $X$ is Asplund,
$f_1$ is Lipschitz continuous and
$f_2$
is lower semicontinuous in a neighbourhood of $\bar x$.
Then, for any $\varepsilon>0$, there exist $x_1,x_2\in X$ with $\|x_i-\bar x\|<\varepsilon$, $|f_i(x_i)-f_i(\bar x)|<\varepsilon$ $(i=1,2)$, such that
$$
\partial (f_1+f_2) (\bar x) \subset \partial f_1(x_1) +\partial f_2(x_2) + \varepsilon\B^\ast.
$$

{\rm (ii) \bf Differentiable sum rule}. Suppose
$f_1$ is Fr\'echet differentiable at $\bx$.
Then,
$$
\partial (f_1+f_2) (\bar x) = \nabla f_1(\bx) +\partial f_2(\bx).
$$

{\rm (iii) \bf Convex sum rule}. Suppose
$f_1$ and $f_2$ are convex and $f_1$ is continuous at a point in $\dom f_2$.
Then,
$$
\partial (f_1+f_2) (\bar x) = \sd f_1(\bx) +\partial f_2(\bx).
$$
\end{lemma}

The first sum rule in the lemma above is known as the \emph{fuzzy} or \emph{approximate} sum rule (Fabian \cite{Fab89}; cf., e.g., \cite[Rule~2.2]{Kru03.1}, \cite[Theorem~2.33]{Mor06.1}) for Fr\'echet subdifferentials in Asplund spaces.
The other two are examples of \emph{exact} sum rules.
They are valid in arbitrary normed spaces (or even locally convex spaces in the case of the last rule).
Rule (ii) can be found, e.g., in \cite[Corollary~1.12.2]{Kru03.1} and \cite[Proposition~1.107]{Mor06.1}.
For rule (iii) we refer the readers to \cite[Theorem~0.3.3]{IofTik79} and
\cite[Theorem~2.8.7]{Zal02}.

The (normalized) \emph{duality mapping} $J$ between a normed space $Y$ and its dual $Y^*$ is defined as (cf. \cite[Definition~3.2.6]{Lucc06})
\begin{gather}\label{J}
J(y):=\left\{y^*\in \mathbb{S}_{Y^*}\mid \langle y^*,y\rangle=\norm{y}\right\},\quad y\in Y.
\end{gather}

% ----------------------------------------------------------------
\section{Error Bounds and Slopes}\label{S1}
In this section, we recall several facts about local error bounds for a special extended-real-valued function $f:X\times Y\to\R_\infty$ on a product of metric spaces in the framework of the general model developed in \cite{Kru15}.
The function is assumed to satisfy $f(\bx,\by)=0$ and
\begin{itemize}
\item [(P1)]
$f(x,y)>0$ if $y\ne\by$,
\item [(P2)]
$\ds\liminf_{f(x,y)\downarrow0}\frac{f(x,y)}{d(y,\by)}>0$.
\end{itemize}
In particular, $y\to\by$ if $f(x,y)\downarrow0$.

Function $f$ is said to have an \emph{error bound} with respect to $x$ at $(\bx,\by)$ with constant $\tau>0$ iff there exists a neighbourhood $U$ of $\bx$, such that
%\marg{$\tau$ can be included in $f$}
\begin{equation}\label{eb}
\tau d(x,S(f)) \le f_+(x,y)\quad \mbox{for all } x\in U,\; y\in Y,
\end{equation}
where
$$S(f):=\{x\in X|\ f(x,\by)\le0\}=\{x\in X|\ f(x,y)\le0 \mbox{ for some } y\in Y\}.$$
The \emph{error bound modulus}
\begin{gather}\label{rr2}
\Er f(\bx,\by):=
\liminf_{\substack{x\to\bx\\f(x,y)>0}} \frac{f(x,y)}{d(x,S(f))}
\end{gather}
coincides with the exact upper bound of all $\tau>0$, such that \eqref{eb} holds true for some neighbourhood $U$ of $\bx$ and provides a quantitative characterization of the \emph{error bound property}.
It is easy to check (cf. \cite[Proposition~3.1]{Kru15}), that
$$\ds\Er f(\bx,\by)=
\liminf_{\substack{x\to\bx,\,y\to\by\\f(x,y)>0}} \frac{f(x,y)}{d(x,S(f))}=
\liminf_{\substack{x\to\bx,\,f(x,y)\downarrow0}} \frac{f(x,y)}{d(x,S(f))}.$$

The case of local error bounds for a function $f:X\to\R_\infty$ of a single variable with $f(\bx)=0$ can be covered by considering its extension $\tilde f:X\times Y\to\R_\infty$ defined, for some $\by\in Y$, by
\begin{equation*}
\tilde f(x,y)=
\begin{cases}
f(x) &\mbox{if } y=\by,
\\
\infty &\mbox{otherwise}.
\end{cases}
\end{equation*}
Conditions (P1) and (P2) are obviously satisfied.

%Several kinds of slopes are used in the characterizations of the error bound property \eqref{eb}.
In the product space $X\times Y$, we are going to use the following asymmetric distance depending on a positive parameter $\rho$:
\begin{equation}\label{drho}
d_\rho((x,y),(u,v)):=\max\{d(x,u),\rho d(y,v)\}.
\end{equation}

Given an $(x,y)\in X\times Y$ with $f(x,y)<\infty$, the local (strong) \emph{slope} \cite{DMT} and \emph{nonlocal slope} \cite{FabHenKruOut12} of $f$ at $(x,y)$ take the following form:
\begin{gather}\label{ls-f}
|\nabla{f}|_{\rho}(x,y):=
\limsup_{\substack{u\to x,\,v\to y\\
(u,v)\ne(x,y)}}
\frac{[f(x,y)-f(u,v)]_+}{d_\rho((u,v),(x,y))},
\\\label{nls-f}
|\nabla{f}|_{\rho}^{\diamond}(x,y):=
\sup_{(u,v)\ne(x,y)}
\frac{[f(x,y)-f_+(u,v)]_+}{d_\rho((x,y),(u,v))}.
\end{gather}
They depend on $\rho$.
We are going to refer to them as the \emph{$\rho$-slope} and \emph{nonlocal $\rho$-slope} of $f$ at $(x,y)$.
In the sequel, superscript `$\diamond$' (diamond) will be used in all constructions derived from \eqref{nls-f} and its analogues to distinguish them from ``conventional'' (local) definitions.

Using \eqref{ls-f} and \eqref{nls-f}, one can define \emph{strict} (limiting) slopes related to the reference point $(\bx,\by)$:
\begin{gather}\label{ss-f}
\overline{|\nabla{f}|}{}^{>}(\bar{x},\by):=
\lim_{\rho\downarrow0}
\inf_{\substack{d(x,\bx)<\rho,\,0<f(x,y)<\rho}}\,
|\nabla{f}|_{\rho}(x,y),
\\\label{mss-f}
\overline{|\nabla{f}|}{}^{>+}(\bar{x},\by):=
\lim_{\rho\downarrow0}
\inf_{\substack{d(x,\bx)<\rho,\,0<f(x,y)<\rho}}\,
\max\left\{|\nabla{f}|_{\rho}(x,y), \frac{f(x,y)}{d(x,\bx)}\right\},
\\\label{uss-f}
\overline{|\nabla{f}|}{}^{\diamond}(\bar{x},\by):=
\lim_{\rho\downarrow0}
\inf_{\substack{d(x,\bx)<\rho,\,0<f(x,y)<\rho}}\,
|\nabla{f}|{}^{\diamond}_{\rho}(x,y).
\end{gather}
They are called, respectively, the \emph{strict outer slope}, \emph{modified strict outer slope}, and \emph{uniform strict outer slope} of $f$ at $(\bx,\by)$; cf. \cite{Kru15,Kru15.2}.
The word ``strict'' reflects the fact that $\rho$-slopes at nearby points contribute to definitions \eqref{ss-f}--\eqref{uss-f} making them analogues of the strict derivative.
The word ``outer'' is used to emphasize that
only points outside the set $S(f)$ are taken into account.
The word ``uniform'' emphasizes the nonlocal (non-limiting) character of $|\nabla{f}|_\rho^{\diamond}(x,y)$ involved in definition \eqref{uss-f}.
Observe that the definition of the modified strict outer slope \eqref{mss-f} contains under $\max$ a nonlocal (when $(x,y)$ is fixed) component $f(x,y)/d(x,\bx)$.

Constants \eqref{ls-f}--\eqref{uss-f} are nonnegative and can be infinite.
In \eqref{ss-f}--\eqref{uss-f}, the usual convention that the infimum of the empty set equals $+\infty$ is in force.
In these definitions, we have not only $x\to\bx$ and $f(x,y)\downarrow0$, but also the metric on $X\times Y$ used in the definitions of the corresponding $\rho$-slopes changing with the contribution of the $y$ component diminishing as $\rho\downarrow0$.

Now suppose that $X$ and $Y$ are normed linear spaces.
In the product space $X\times Y$,
we consider the following $\rho$-norm $\|\cdot\|_\rho$ being the realization of the $\rho$-metric \eqref{drho}:
\begin{gather*}\label{nrho}
\|(u,v)\|_\rho=\max\{\|u\|,\rho\|v\|\},\quad (u,v)\in X\times Y.
\end{gather*}
The corresponding dual norm (we keep for it the same notation $\|\cdot\|_\rho$) is of the form:
\begin{gather}\label{nrho*}
\|(u^*,v^*)\|_\rho=\|u^*\|+\rho\iv\|v^*\|,\quad (u^*,v^*)\in X^*\times Y^*.
\end{gather}

One can define subdifferential counterparts of the local slopes \eqref{ls-f}, \eqref{ss-f}, and \eqref{mss-f}:
the \emph{subdifferential $\rho$-slope} \begin{equation}\label{sds-f}
|\sd{f}|_{\rho}(x,y):=
\inf_{\substack{(x^*,y^*)\in\sd f(x,y),\,
\|y^*\|<\rho}} \|x^*\|
\end{equation}
of $f$ at $(x,y)$ ($f(x,y)<\infty$) and the \emph{strict outer} and, respectively, \emph{modified strict outer subdifferential slopes}
\begin{gather}\label{ssds-f}
\overline{|\sd{f}|}{}^{>}(\bar{x},\by):=
\lim_{\rho\downarrow0}
\inf_{\substack{\|x-\bx\|<\rho,\,0<f(x,y)<\rho}}\,
|\sd{f}|_{\rho}(x,y),
\\\label{mssds-f}
\overline{|\sd{f}|}{}^{>+}(\bar{x},\by):=
\lim_{\rho\downarrow0}
\inf_{\substack{\|x-\bx\|<\rho,\,0<f(x,y)<\rho}}
\max\left\{|\sd{f}|_{\rho}(x,y), \frac{f(x,y)}{\|x-\bx\|}\right\}
\end{gather}
of $f$ at
$(\bx,\by)$.

Note that, unlike the $\rho$-slope \eqref{ls-f}, the subdifferential $\rho$-slope \eqref{sds-f} is not the realization of the subdifferential slope \cite{FabHenKruOut10} for the case of a function of two variables.

The next proposition coming from \cite[Proposition~2]{Kru15.2} summarizes the relationships between the slopes.

\begin{proposition}[Relationships between slopes]\label{nc}
\begin{enumerate}
\item
$\ds|\nabla{f}|_{\rho}^{\diamond}(x,y)\ge
\max\left\{|\nabla{f}|_{\rho}(x,y), \frac{f(x,y)}{d_\rho((x,y),(\bx,\by))}\right\}$
for all $\rho>0$ and all $(x,y)\in X\times Y$ with $0<f(x,y)<\infty$;
\item
$\ds\overline{|\nabla{f}|}{}^{\diamond}(\bar{x},\by)\ge
\overline{|\nabla{f}|}{}^{>+}(\bx,\by)\ge
\overline{|\nabla{f}|}{}^{>}(\bx,\by)$.
\cnta
\end{enumerate}
%If $f$ is convex, then {\rm (i)} and {\rm (ii)} hold as equalities.\\
Suppose $X$ and $Y$ are normed linear spaces.
\begin{enumerate}
\cntb
\item
$\ds\overline{|\sd{f}|}{}^{>+}(\bx,\by)\ge
\overline{|\sd{f}|}{}^{>}(\bx,\by)$;
\item
$|\nabla{f}|_{\rho}(x,y)\le
|\sd{f}|_{\rho^2}(x,y)+\rho$ for all $\rho>0$ and all $(x,y)\in X\times Y$ with $f(x,y)<\infty$;
\item
$\overline{|\nabla{f}|}{}^{>}(\bar{x},\by)\le
\overline{|\sd{f}|}{}^{>}(\bx,\by)$ and
$\overline{|\nabla{f}|}{}^{>+}(\bar{x},\by)\le
\overline{|\sd{f}|}{}^{>+}(\bx,\by)$;
\item\label{nc.5}
$\overline{|\nabla{f}|}{}^{>}(\bar{x},\by)=
\overline{|\sd{f}|}{}^{>}(\bx,\by)$ and $\overline{|\nabla{f}|}{}^{>+}(\bar{x},\by)=
\overline{|\sd{f}|}{}^{>+}(\bx,\by)$,\\
provided that one of the following conditions is satisfied:
\begin{enumerate}
\item
$X$ and $Y$ are Asplund and $f_+$ is \lsc\ near $(\bar{x},\by)$;
\item
$f$ is convex;
in this case {\rm (i)} and {\rm (ii)} also hold as equalities;
\item
$f$ is Fr\'echet differentiable near $(\bar{x},\by)$ except $(\bx,\by)$;
\item
$f=f_1+f_2$, where $f_1$ is convex near $(\bar{x},\by)$ and $f_2$ is Fr\'echet differentiable near $(\bar{x},\by)$ except $(\bx,\by)$.
\end{enumerate}
\end{enumerate}
\end{proposition}

\begin{remark}\label{R4}
One of the main tools in the proof of inequalities
\begin{align*}
\overline{|\nabla{f}|}{}^>(\bar{x},\by)
\ge\overline{|\partial{f}|}{}^>(\bar{x},\by),\qquad
\overline{|\nabla{f}|}{}^{>+}(\bar{x},\by)
\ge\overline{|\partial{f}|}{}^{>+}(\bar{x},\by)
\end{align*}
in item (a) of part (\ref{nc.5}) of the above proposition, which is crucial for the subdifferential sufficient error bound criteria, is the fuzzy sum rule (Lemma~\ref{l02}) for Fr\'echet subdifferentials in Asplund spaces.
It is possible to extend these inequalities to general Banach spaces by
replacing Fr\'echet subdifferentials with some other subdifferentials on the given space satisfying a certain set of natural properties including a kind of (fuzzy or exact) sum rule.
One can use for that purpose \emph{Ioffe approximate} or \emph{Clarke} subdifferentials.
Note that the opposite inequalities in part (v) are specific for Fr\'echet subdifferentials and fail in general for other types of subdifferentials.
\end{remark}

The uniform strict outer slope \eqref{uss-f} provides the necessary and sufficient characterization of error bounds \cite[Theorem 4.1]{Kru15}.

\begin{theorem}\label{3T1}
\begin{enumerate}
\item
$\Er f(\bx,\by)\le
\overline{|\nabla{f}|}{}^{\diamond}(\bar{x},\by)$;
\item
if $X$ and $Y$ are complete and $f_+$ is lower semicontinuous
near $(\bar{x},\by)$,
then
$\Er f(\bx,\by)=
\overline{|\nabla{f}|}{}^{\diamond}(\bar{x},\by)$.
\end{enumerate}
\end{theorem}

\begin{remark}\label{rm2}
The nonlocal $\rho$-slope \eqref{nls-f} depends on the choice of $\rho$-metric on the product space.
If instead of the metric $d_\rho$, defined by \eqref{drho}, one employs in \eqref{nls-f} the sum-type parametric metric $d_\rho^1$, defined by
\begin{gather*}
d_\rho^1((x,y),(u,v)):=d(x,u)+\rho d(y,v),
\end{gather*}
it will produce a different number.
We say that a $\rho$-metric $d'_\rho$ on $X\times Y$ is admissible iff $d_\rho\le d'_\rho\le d^1_\rho$.
Thanks to \cite[Proposition~4.2]{Kru15}, Theorem~\ref{3T1} is invariant on the choice of an admissible $\rho$-metric.
\end{remark}

Thanks to Theorem~\ref{3T1} and Proposition~\ref{nc}, one can formulate several quantitative and qualitative criteria of the error bound property in terms of various slopes discussed above; cf. \cite[Corollaries~1 and 2]{Kru15.2}.

\section{Nonlinear Metric Subregularity}\label{S3}

From now on, $F:X\rightrightarrows Y$ is a set-valued mapping between metric spaces and $(\bx,\by)\in\gph F$.
We are targeting several versions of the metric subregularity property, the main tool being the error bound criteria discussed in the previous section.

\subsection{Metric $f$-subregularity and metric $g$-subregularity}

Alongside the set-valued mapping $F$, we consider an extended-real-valued function $f:X\times Y\to\R_\infty$, satisfying the assumptions made in Section~\ref{S1}, i.e., $f(\bx,\by)=0$ and properties (P1) and (P2).
Additionally, we assume that $f$ takes only nonnegative values, i.e., $f:X\times Y\to\R_+\cup\{+\infty\}$, and
\begin{itemize}
\item [(P3)]
$f(x,y)=0$ iff $y=\by$ and $x\in F^{-1}(\by)$).
\end{itemize}
Hence, $S(f)=F^{-1}(\by)$.

We say that $F$ is \emph{metrically $f$-subregular} at $(\bx,\by)$ with constant $\tau>0$ iff there exists a neighbourhood $U$ of $\bx$, such that
\begin{equation}\label{MR2}
\tau d(x,F^{-1}(\by))\le f(x,y)\quad \mbox{for all } x\in U,\; y\in Y.
\end{equation}
Metric $f$-subregularity property can be characterized using the following (possibly infinite) constant:
\begin{gather}\label{CMR-g}
\sr_f[F](\bx,\by):=
\liminf_{\substack{x\to\bx\\x\notin F\iv(\by),\,y\in Y}} \frac{f(x,y)}{d(x,F^{-1}(\by))},
\end{gather}
which coincides with the supremum of all positive $\tau$, such that \eqref{MR2} holds for some $U$.

In the special case when $f$ is given by
$$f(x,y):=
\begin{cases}
d(y,\by) & \text{if } (x,y)\in\gph F,\\
+\infty & \text{otherwise},
\end{cases}
$$
conditions (P1)--(P3) are trivially satisfied and the metric $f$-subregularity reduces to the conventional metric subregularity (cf., e.g., \cite{Mor06.1,RocWet98,DonRoc14}).

In general,
property \eqref{MR2} is exactly the error bound property \eqref{eb} for the function $f$ while constant \eqref{CMR-g} coincides with \eqref{rr2}.
Hence, the main characterization of the metric $f$-subregularity is given by the above Theorem~\ref{3T1}, which yields a series of sufficient criteria in terms of various kinds of local slopes; cf. \cite[Corollaries~1 and 2]{Kru15.2}.
Note that one can always suppose $\tau=1$ in \eqref{MR2}: it is sufficient to replace function $f$ with $f/\tau$.
We keep $\tau$ in the definitions in this and subsequent sections for the purpose of uniformity of the presentation.

In the rest of the section, we consider a special case of metric $f$-subregularity of $F$ with $f$ defined by
\begin{gather}\label{gF2}
f(x,y)=g(y)+ i_{\gph F}(x,y),\quad x\in X,\; y\in Y,
\end{gather}
where $g:Y\to\R_+$ and $i_{\gph F}$ is the indicator function of $\gph F$: $i_{\gph F}(x,y)=0$ if $(x,y)\in\gph F$ and $i_{\gph F}(x,y)=\infty$ otherwise.

We say that $F$ is \emph{metrically $g$-subregular} at $(\bx,\by)$ with constant $\tau>0$ iff there exists a neighbourhood $U$ of $\bx$, such that
\begin{equation}\label{5MR2}
\tau d(x,F^{-1}(\by))\le g(y)\quad \mbox{for all } x\in U,\; y\in F(x).
\end{equation}

We will assume that $g(\by)=0$, $g$ is continuous at $\by$, locally Lipschitz continuous on $Y\setminus\{\by\}$ and satisfies the following properties:
\begin{itemize}
\item [(P1$'$)]
$g(y)>0$ if $y\ne\by$,
\item [(P2$'$)]
$\ds\liminf_{g(y)\downarrow0}\frac{g(y)}{d(y,\by)}>0$.
\end{itemize}
These properties obviously imply the corresponding properties (P1) and (P2) for the function $f$ defined by \eqref{gF2}.
Property (P3) is satisfied automatically.
Observe that, thanks to the continuity of $g$,
property (P2$'$) entails the equivalence $$g(y)\downarrow0\quad\Leftrightarrow\quad y\to\by,$$
which leads to simplifications in some definitions.

Metric $g$-subregularity \eqref{5MR2} can be equivalently characterized using the following constant being the realization of \eqref{CMR-g}:
\begin{gather}\label{5CMR-g}
\sr_g[F](\bx,\by):=
\liminf_{\substack{x\to\bx\\x\notin F\iv(\by),\,y\in F(x)}} \frac{g(y)}{d(x,F^{-1}(\by))}.
\end{gather}

\subsection{Primal space and subdifferential slopes}
The $\rho$-slope \eqref{ls-f} and nonlocal $\rho$-slope \eqref{nls-f} of $f$ at $(x,y)\in\gph F$ in the current setting can be rewritten as follows:
\begin{gather}\label{5ls}
|\nabla{F}|_{g,\rho}(x,y):=
\limsup_{\substack{(u,v)\to(x,y),\,
(u,v)\ne(x,y)\\(u,v)\in\gph F}}
\frac{[g(y)-g(v)]_+}{d_\rho((u,v),(x,y))},
\\\label{5nls}
|\nabla{F}|_{g,\rho}^{\diamond}(x,y):=
\sup_{\substack{(u,v)\ne(x,y)\\(u,v)\in\gph F}}
\frac{[g(y)-g(v)]_+}{d_\rho((u,v),(x,y))}.
\end{gather}
We will call the above constants, respectively, the \emph{$(g,\rho)$-slope} and \emph{nonlocal $(g,\rho)$-slope} of $F$ at $(x,y)$.

The strict slopes \eqref{ss-f}--\eqref{uss-f} produce the following definitions:
\begin{gather}\label{5ss}
\overline{|\nabla{F}|}{}_g(\bar{x},\by):=
\lim_{\rho\downarrow0}
\inf_{\substack{d(x,\bx)<\rho,\,d(y,\by)<\rho\\
(x,y)\in\gph F,\,x\notin F\iv(\by)}}\,
|\nabla{F}|_{g,\rho}(x,y),
\\\label{5mss}
\overline{|\nabla{F}|}{}_g^{+}(\bar{x},\by):=
\lim_{\rho\downarrow0}
\inf_{\substack{d(x,\bx)<\rho,\,d(y,\by)<\rho\\
(x,y)\in\gph F,\,x\notin F\iv(\by)}}\,
\max\left\{|\nabla{F}|_{g,\rho}(x,y), \frac{g(y)}{d(x,\bx)}\right\},
\\\label{5uss}
\overline{|\nabla{F}|}{}^{\diamond}_g(\bar{x},\by):=
\lim_{\rho\downarrow0}
\inf_{\substack{d(x,\bx)<\rho,\,d(y,\by)<\rho\\
(x,y)\in\gph F,\,x\notin F\iv(\by)}}\,
|\nabla{F}|{}^{\diamond}_{g,\rho}(x,y).
\end{gather}
They are called, respectively,
the \emph{strict $g$-slope}, \emph{modified strict $g$-slope}, and \emph{uniform strict $g$-slope} of $F$ at $(\bx,\by)$.

The continuity of $g$ was taken into account when writing down \eqref{5ss}--\eqref{5uss}.
Note that conditions $(x,y)\in\gph F$ and $x\notin F\iv(\by)$ imply $y\ne\by$.

%\subsection{Subdifferential slopes}
If $X$ and $Y$ are normed linear spaces, the subdifferential slopes of $f$ are defined by \eqref{sds-f}, \eqref{ssds-f}, and \eqref{mssds-f}.
In the current setting, the last two constants take the following form:
\begin{gather}\label{5ssds}
\overline{|\sd{f}|}{}^>(\bar{x},\by)=
\lim_{\rho\downarrow0}
\inf_{\substack{\|x-\bx\|<\rho,\,\|y-\by\|<\rho\\
(x,y)\in\gph F,\,
x\notin F\iv(\by)}}\,
|\sd{f}|_{\rho}(x,y),
\\\label{5mssds}
\overline{|\sd{f}|}{}^{>+}(\bar{x},\by)=
\lim_{\rho\downarrow0}
\inf_{\substack{\|x-\bx\|<\rho,\,\|y-\by\|<\rho\\
(x,y)\in\gph F,\,
x\notin F\iv(\by)}}\,
\max\left\{|\sd{f}|_{\rho}(x,y), \frac{g(y)}{\|x-\bx\|}\right\}.
\end{gather}

The next statement is a consequence of Proposition~\ref{nc}.

\begin{proposition}[Relationships between slopes]\label{5P2}
\begin{enumerate}
\item
$\ds|\nabla{F}|_{g,\rho}^{\diamond}(x,y)\ge
\max\left\{|\nabla{F}|_{g,\rho}(x,y), \frac{g(y)}{d_\rho((x,y),(\bx,\by))}\right\}$
for all $\rho>0$ and all $(x,y)\in\gph F$;
\item
$\ds\overline{|\nabla{F}|}{}_g^{\diamond}(\bar{x},\by)\ge
\overline{|\nabla{F}|}{}_{g}^+(\bx,\by)\ge
\overline{|\nabla{F}|}{}_{g}(\bx,\by)$.
\cnta
\end{enumerate}
%If $F$ and $g$ are convex, then {\rm (i)} and {\rm (ii)} hold as equalities.\\
Suppose $X$ and $Y$ are normed spaces.
\begin{enumerate}
\cntb
\item
$|\nabla{F}|_{g,\rho}(x,y)\le
|\sd{f}|_{\rho^2}(x,y)+\rho$ for all $\rho>0$ and $(x,y)\in\gph F$;
\item
$\overline{|\nabla{F}|}{}_g(\bar{x},\by)\le
\overline{|\sd{f}|}{}^{>}(\bx,\by)$ and $\overline{|\nabla{F}|}{}_g^+(\bar{x},\by)\le
\overline{|\sd{f}|}{}^{+>}(\bx,\by)$;
\item\label{5P2.5}
$\overline{|\nabla{F}|}{}_g(\bar{x},\by)=
\overline{|\sd{f}|}{}^{>}(\bx,\by)$ and $\overline{|\nabla{F}|}{}_g^+(\bar{x},\by)=
\overline{|\sd{f}|}{}^{+>}(\bx,\by)$,\\
provided that one of the following conditions is satisfied:
\begin{enumerate}
\item
$X$ and $Y$ are Asplund and $\gph F$ is locally closed near $(\bar{x},\by)$;
\item
$F$ is convex and $g$ is either convex or Fr\'echet differentiable near $\by$ except $\by$;
if both $F$ and $g$ are convex, then {\rm (i)} and {\rm (ii)} also hold as equalities.
\end{enumerate}
\end{enumerate}
\end{proposition}

Observe that condition (b) in part \eqref{5P2.5} of Proposition~\ref{5P2} ensures that function $f$ defined by \eqref{gF2} satisfies either condition (b) or condition (d) in part {\rm (\ref{nc.5})} of Proposition~\ref{nc}.

Now we define the \emph{subdifferential $(g,\rho)$-slope} and \emph{approximate subdifferential $(g,\rho)$-slo\-pe} of $F$ at $(x,y)\in\gph F$:
\begin{gather}\label{5f02}
|\sd{F}|_{g,\rho}(x,y):=
\inf_{\substack{
x^*\in D^*F(x,y)(\sd g(y)+\rho\B^*)}}
\|x^*\|,
\\\label{5f02a}
|\sd{F}|^a_{g,\rho}(x,y):=
\liminf_{\substack{y'\to y}}\
\inf_{\substack{
x^*\in D^*F(x,y)(\sd g(y')+\rho\B^*)}}
\|x^*\|
\end{gather}
and use them to define strict subdifferential $g$-slopes:
\begin{gather}\label{5f03}
\overline{|\sd{F}|}{}_g(\bar{x},\by):=
\lim_{\rho\downarrow0}
\inf_{\substack{\|x-\bx\|<\rho,\,\|y-\by\|<\rho\\
(x,y)\in\gph F,\,x\notin F\iv(\by)}}\,
|\sd{F}|_{g,\rho}(x,y),
\\\label{5f03a}
\overline{|\sd{F}|}{}^a_g(\bar{x},\by):=
\lim_{\rho\downarrow0}
\inf_{\substack{\|x-\bx\|<\rho,\,\|y-\by\|<\rho\\
(x,y)\in\gph F,\,x\notin F\iv(\by)}} |\sd{F}|^a_{g,\rho}(x,y),
\\\label{5mf03}
\overline{|\sd{F}|}{}_g^{+}(\bar{x},\by):=
\lim_{\rho\downarrow0}
\inf_{\substack{\|x-\bx\|<\rho,\,\|y-\by\|<\rho\\
(x,y)\in\gph F,\,
x\notin F\iv(\by)}}\,
\max\left\{|\sd{F}|_{g,\rho}(x,y), \frac{g(y)}{\|x-\bx\|}\right\},
\\\label{5mf03a}
\overline{|\sd{F}|}{}_g^{a+}(\bar{x},\by):=
\lim_{\rho\downarrow0}
\inf_{\substack{\|x-\bx\|<\rho,\,\|y-\by\|<\rho\\
(x,y)\in\gph F,\,
x\notin F\iv(\by)}}\,
\max\left\{|\sd{F}|^a_{g,\rho}(x,y), \frac{g(y)}{\|x-\bx\|}\right\}.
\end{gather}
They are called, respectively,
the \emph{strict subdifferential $g$-slope}, \emph{approximate strict subdifferential $g$-slope}, \emph{modified strict subdifferential $g$-slope}, and \emph{modified approximate strict subdifferential $g$-slope} of $F$ at $(\bx,\by)$.

The next proposition gives relationships between the subdifferential slopes \eqref{sds-f}, \eqref{5ssds}--\eqref{5mf03a}.

\begin{proposition}[Relationships between subdifferential slopes]\label{5P3}
Suppose $X$ and $Y$ are normed spaces and $f$ is given by \eqref{gF2}.
\begin{enumerate}
\item
$|\sd{F}|^a_{g,\rho}(x,y) \le|\sd{F}|_{g,\rho}(x,y)$ for all $\rho>0$ and $(x,y)\in\gph F$;
\item
$\overline{|\sd{F}|}{}^a_g(\bar{x},\by) \le\overline{|\sd{F}|}{}_g(\bar{x},\by) \le\overline{|\sd{F}|}{}^+_g(\bar{x},\by)$ and\\ $\overline{|\sd{F}|}{}^{a}_g(\bar{x},\by) \le\overline{|\sd{F}|}{}^{a+}_g(\bar{x},\by) \le\overline{|\sd{F}|}{}^+_g(\bar{x},\by)$.
\cnta
\end{enumerate}
If $X$ and $Y$ are Asplund and $\gph F$ is locally closed near $(\bar{x},\by)$, then
\begin{enumerate}
\cntb
\item
$\ds|\sd{f}|_{\rho}(x,y)
\ge \liminf_{\substack{ (x',y')\to(x,y),\,y''\to y\\(x',y')\in\gph F}}\; \inf_{\substack{
x^*\in D^*F(x',y')(\sd g(y'')+\rho\B^*)}}
\|x^*\|$
for all $\rho>0$ and $(x,y)\in\gph F$ near $(\bar{x},\by)$;
\item
$\ds\overline{|\sd{f}|}{}^>(\bar{x},\by) \ge\overline{|\sd{F}|}{}^a_g(\bar{x},\by)$ and $\ds\overline{|\sd{f}|}{}^{>+}(\bar{x},\by) \ge\overline{|\sd{F}|}{}^{a+}_g(\bar{x},\by)$.
\cnta
\end{enumerate}
If either
$F$ and $g$ are convex or $g$ is Fr\'echet differentiable near $\by$ except $\by$, then
\begin{enumerate}
\cntb
\item
$\ds|\sd{f}|_{\rho}(x,y)
= |\sd{F}|_{g,\rho}(x,y)$ for all $\rho>0$ and $(x,y)\in\gph F$ near $(\bar{x},\by)$ with $y\ne\by$;
\item
$\ds\overline{|\sd{f}|}{}^>(\bar{x},\by) =\overline{|\sd{F}|}{}_g(\bar{x},\by)$ and $\ds\overline{|\sd{f}|}{}^{>+}(\bar{x},\by) =\overline{|\sd{F}|}{}^+_g(\bar{x},\by)$.
\end{enumerate}
\end{proposition}

\begin{proof}
Inequalities (i) and (ii) follow directly from the definitions.

(iii) Let $\rho>0$ and $(x,y)\in\gph F$ near $(\bar{x},\by)$ be given, such that $\gph F$ is locally closed and $g$ is \lsc\ near $(x,y)$.
The \emph{fuzzy sum rule} (Lemma~\ref{l02}) is applicable to $f$:
for any $\eps>0$,
\begin{gather}\label{5f01}
\sd f(x,y)
\subset\bigcup_{\substack{
\|(x',y')-(x,y)\|<\eps,\
(x',y')\in\gph F\\
(x^*,y^*)\in N_{\gph F}(x',y')\\
\|y''-y\|<\eps,\,v^*\in\sd g(y'')}} \{x^*,y^*+v^*\}+\eps\B_{X^*\times Y^*}.
\end{gather}
By definition \eqref{sds-f},
\begin{align*}
|\sd{f}|_{\rho}(x,y)
&\ge\lim_{\eps\downarrow0}\Biggl(\inf_{\substack{
\|(x',y')-(x,y)\|<\eps,\
(x',y')\in\gph F\\
(x^*,y^*)\in N_{\gph F}(x',y')\\
\|y''-y\|<\eps,\,v^*\in\sd g(y'')\\ \|y^*+v^*\|<\rho}} \|x^*\|-\eps\Biggr)
\\
&=\lim_{\eps\downarrow0}\inf_{\substack{
\|(x',y')-(x,y)\|<\eps,\
(x',y')\in\gph F\\
x^*\in D^*F(x',y')(y^*)\\
\|y''-y\|<\eps,\,v^*\in\sd g(y'')\\ \|y^*-v^*\|<\rho}} \|x^*\|
\\
&=\lim_{\eps\downarrow0}\inf_{\substack{
\|(x',y')-(x,y)\|<\eps,\
(x',y')\in\gph F\\
x^*\in D^*F(x',y')(\sd g(y'')+\rho\B^*)\\
\|y''-y\|<\eps}} \|x^*\|
\\
&=\lim_{\eps\downarrow0}\inf_{\substack{
\|(x',y')-(x,y)\|<\eps\\
(x',y')\in\gph F\\
\|y''-y\|<\eps}}\inf_{\substack{x^*\in D^*F(x',y')(\sd g(y'')+\rho\B^*)}} \|x^*\|
\\
&=\liminf_{\substack{ (x',y')\to(x,y),\,y''\to y\\(x',y')\in\gph F}}\; \inf_{\substack{x^*\in D^*F(x',y')(\sd g(y'')+\rho\B^*)}} \|x^*\|.
\end{align*}

(iv) By \eqref{5ssds} and (iii),
\begin{align*}
\overline{|\sd{f}|}{}^>(\bar{x},\by)
&\ge\lim_{\rho\downarrow0}
\inf_{\substack{\|(x,y)-(\bx,\by)\|<\rho\\
(x,y)\in\gph F,\,
x\notin F\iv(\by)}}\lim_{\eps\downarrow0}
\inf_{\substack{\|(x',y')-(x,y)\|<\eps,\, \|y''-y\|<\eps\\x^*\in D^*F(x',y')(\sd g(y'')+\rho\B^*)\\ (x',y')\in\gph F}}
\|x^*\|.
\end{align*}
For a fixed $(x,y)$ with $x\notin F\iv(\by)$ and a sufficiently small $\eps>0$, it holds $B_\eps(x)\cap F\iv(\by)=\emptyset$ and $\|(x,y)-(\bx,\by)\|+\eps<\rho$.
Besides, $\|y''-y'\|\le \|y''-y\|+\|y'-y\|<2\eps$.
Hence,
\begin{align*}
\overline{|\sd{f}|}{}^>(\bar{x},\by)
&\ge\lim_{\rho\downarrow0}
\inf_{\substack{\|(x',y')-(\bx,\by)\|<\rho\\
(x',y')\in\gph F,\,
x'\notin F\iv(\by)}}\,
\lim_{\eps\downarrow0}
\inf_{\substack{\|y''-y'\|<2\eps\\x^*\in D^*F(x',y')(\sd g(y'')+\rho\B^*)}}\
\|x^*\|
\\&
=\lim_{\rho\downarrow0}
\inf_{\substack{\|(x',y')-(\bx,\by)\|<\rho\\
(x',y')\in\gph F,\,
x'\notin F\iv(\by)}}\,
|\sd{F}|^a_{g,\rho}(x',y')
=\overline{|\sd{F}|}{}^a_g(\bar{x},\by).
\end{align*}
The proof of the other inequality goes along the same lines.

(v) The proof is similar to that of (iii).
Instead of the fuzzy sum rule, one can use either the differentiable rule (Lemma~\ref{l02}(ii)) or the convex sum rule (Lemma~\ref{l02}(iii)) to write down, for all $(x,y)\in\gph F$ near $(\bx,\by)$ with $y\ne\by$, the representation:
\begin{gather*}\label{5f05}
\sd f(x,y)
=\bigcup_{(x^*,y^*)\in N_{\gph F}(x,y)} \{x^*,\sd g(y)+y^*\},
\end{gather*}
where $\sd g(y)=\{\nabla g(y)\}$ if $g$ is differentiable at $y$.
By definition \eqref{sds-f},
\begin{align*}
|\sd{f}|_{\rho}(x,y)
&=\inf_{\substack{
(x^*,y^*)\in N_{\gph F}(x,y)\\
v^*\in\sd g(y),\, \|y^*+v^*\|<\rho}} \|x^*\|
=\inf_{\substack{
x^*\in D^*F(x,y)(\sd g(y)+\rho\B^*)}} \|x^*\| =|\sd{F}|_{g,\rho}(x,y).
\end{align*}

(vi) follows from (v) in view of representations \eqref{5ssds} and \eqref{5mssds}.
\qed\end{proof}

Proposition~\ref{5P3} allows one to eliminate subdifferential slopes of $f$ from the estimates in Proposition~\ref{5P2}.

\begin{proposition}[Relationships between slopes]\label{5P6}
Suppose $X$ and $Y$ are normed spaces.
\begin{enumerate}
\item
$\overline{|\nabla{F}|}{}_g(\bar{x},\by)\ge
\overline{|\sd{F}|}{}^a_{g}(\bx,\by)$ and $\overline{|\nabla{F}|}{}_g^+(\bar{x},\by)\ge
\overline{|\sd{F}|}{}_{g}^{a+}(\bx,\by)$,\\
provided that $X$ and $Y$ are Asplund and $\gph F$ is locally closed near $(\bar{x},\by)$;
\item
$\overline{|\nabla{F}|}{}_g(\bar{x},\by)=
\overline{|\sd{F}|}{}_{g}(\bx,\by)$ and $\overline{|\nabla{F}|}{}_g^+(\bar{x},\by)=
\overline{|\sd{F}|}{}_{g}^{+}(\bx,\by)$,\\
provided that
$g$ is Fr\'echet differentiable near $\by$ except $\by$ and one of the following conditions is satisfied:
\begin{enumerate}
\item
$X$ and $Y$ are Asplund and $\gph F$ is locally closed near $(\bar{x},\by)$;
\item
$F$ is convex;
\end{enumerate}
\item
$\ds\overline{|\nabla{F}|}{}^{\diamond}_g(\bar{x},\by)=
\overline{|\nabla{F}|}{}^{+}_g(\bx,\by)=
\overline{|\nabla{F}|}{}_g(\bx,\by)= \overline{|\sd{F}|}{}^{+}_g(\bx,\by)= \overline{|\sd{F}|}{}_g(\bx,\by)$,
provided that
$F$ and $g$ are convex.
\end{enumerate}
\end{proposition}

\subsection{Limiting $g$-coderivatives}
In finite dimensions, strict subdifferential $g$-slopes \eqref{5f03} and \eqref{5f03a} can be equivalently expressed in terms of certain kinds of limiting coderivatives.

The \emph{limiting outer $g$-coderivative} $\overline{D}{}^{*>}_gF(\bx,\by)$ and the \emph{approximate limiting outer $g$-coderivative} $\overline{D}{}^{*>a}_gF(\bx,\by)$ of $F$ at $(\bx,\by)$ are defined by their graphs as follows:
\begin{align}\notag
\gph\overline{D}{}^{*>}_gF (\bx,\by):=& \{(y^*,x^*)\in Y^*\times X^*\mid
\exists (x_k,y_k,x^*_k,y^*_k,v^*_k)\subset X\times Y\times X^*\times Y^*\times Y^*\;\mbox{such that}
\\\notag
&(x_k,y_k)\in\gph F,\;x_k\notin F\iv(\by),\;
(y_k^*,x_k^*)\in\gph{D}{}^{*}F(x_k,y_k),\; \\\notag
&v^*_k\in\sd g(y_k),\;
(x_k,y_k)\to(\bx,\by),\;y^*_k-v^*_k\to0,\; \|y^*\|x^*_k\to x^*,
\\\label{D*1}
&\mbox{if}\; y^*\ne0,\; \mbox{then either}\; y_k^*\ne0\; (\forall k\in\N)\;\mbox{and}\; \frac{y^*_k}{\|y_k^*\|}\to\frac{y^*}{\|y^*\|},\;
\mbox{or}\; y_k^*=0\; (\forall k\in\N)\},
\\\notag
\gph\overline{D}{}^{*>a}_gF (\bx,\by):=& \{(y^*,x^*)\in Y^*\times X^*\mid
\exists (x_k,y_k,x^*_k,y^*_k,v^*_k)\subset X\times Y\times X^*\times Y^*\times Y^*\;\mbox{such that}
\\\notag
&(x_k,y_k)\in\gph F,\;x_k\notin F\iv(\by),\;
(y_k^*,x_k^*)\in\gph{D}{}^{*}F(x_k,y_k),\;
\\\notag
&v^*_k\in\overline{\sd}g(y_k),\;
(x_k,y_k)\to(\bx,\by),\;y^*_k-v^*_k\to0,\; \|y^*\|x^*_k\to x^*,
\\\label{D*3}
&\mbox{if}\; y^*\ne0,\; \mbox{then either}\; y_k^*\ne0\; (\forall k\in\N)\;\mbox{and}\; \frac{y^*_k}{\|y_k^*\|}\to\frac{y^*}{\|y^*\|},\;
\mbox{or}\; y_k^*=0\; (\forall k\in\N)\},
\end{align}
where
\begin{align*}
\overline{\sd}g(v):= \{v^*\in & Y^*\mid \exists (v_k,v^*_k)\to(v,v^*)
\quad\mbox{such that}\quad v^*_k\in\sd g(v_k)\}
\end{align*}
is the \emph{limiting subdifferential} of $g$ at $y$; cf. \cite{Mor06.1,RocWet98}.

The sets defined by \eqref{D*1} and \eqref{D*3} are closed cones in $X\times Y$.
Hence, all limiting outer $g$-coderivatives are closed positively homogeneous \SVM s.

\begin{proposition}\label{ows}
Suppose $X$ and $Y$ are finite dimensional normed linear spaces.
The following equalities hold true.
\begin{enumerate}
\item
$\overline{|\sd{F}|}{}_g(\bar{x},\by) =\inf\limits_ {\substack{x^*\in\overline{D}{}^{*>}_gF(\bx,\by) (\Sp^*_{Y^*})}} \|x^*\|$,
\item
$\overline{|\sd{F}|}{}^{a}_g(\bar{x},\by) =\inf\limits_ {\substack{x^*\in\overline{D}{}^{*>a}_gF(\bx,\by) (\Sp^*_{Y^*})}} \|x^*\|$.
\end{enumerate}
\end{proposition}

\begin{proof}
(i) Let $(y^*,x^*)\in\gph\overline{D}{}^{*>}_gF(\bx,\by)$, $\|y^*\|=1$, and $\rho>0$.
Choose a sequence $(x_k,y_k,x^*_k,y^*_k,v^*_k,\al_k)$ corresponding to $(y^*,x^*)$ in accordance with definition~\eqref{D*1}.
Then, for a sufficiently large $k$, it holds $\|x_k-\bx\|<\rho$, $\|y_k-\by\|<\rho$, $(x_k,y_k)\in\gph F$, $x_k\notin F\iv(\by)$, $y_k^*\in\sd g(y_k)+\rho\B^*$, $x_k^*\in D^*F(x_k,y_k)(y_k^*)$, and $\|x_k^*-x^*\|<\rho$.
Hence, by \eqref{5f02} and \eqref{5f03}, $\overline{|\sd{F}|}{}_g(\bar{x},\by) \le\|x_k^*\|<\|x^*\|+\rho$, and consequently
$$
\overline{|\sd{F}|}{}_g(\bar{x},\by) \le\inf\limits_ {\substack{x^*\in\overline{D}{}^{*>}_gF(\bx,\by) (\Sp^*_{Y^*})}} \|x^*\|.
$$

Conversely, by definitions \eqref{5f02} and \eqref{5f03}, there exist sequences $(x_k,y_k)\to(\bx,\by)$ with $(x_k,y_k)\in\gph F$, $x_k\notin F\iv(\by)$ and $(x_k^*,y_k^*,v_k^*)\in X^*\times Y^*\times Y^*$ with $(y_k^*,x_k^*)\in\gph{D}{}^{*}F(x_k,y_k)$,  $v^*_k\in\sd g(y_k)$, such that $y^*_k-v^*_k\to0$ and $\|x_k^*\|\to\overline{|\sd{F}|}{}_g(\bar{x},\by)$.
Without loss of generality, $x_k^*\to x^*\in X^*$ and either $y_k^*\ne0$ for all $k\in\N$, or $y_k^*=0$ for all $k\in\N$. In the first case, we can assume that $y_k^*/\|y_k^*\|\to y^*\in\Sp_{Y^*}^*$, and consequently, by definition \eqref{D*1}, $(x^*,y^*)\in\gph\overline{D}{}^{*>}_gF(\bx,\by)$.
In the second case, $(x^*,y^*)\in\gph\overline{D}{}^{*>}_gF(\bx,\by)$ for any $y^*\in Y^*$.
Hence,
$$
\overline{|\sd{F}|}{}_g(\bar{x},\by)=\|x^*\| \ge\inf\limits_ {\substack{x^*\in\overline{D}{}^{*>}_gF(\bx,\by) (\Sp^*_{Y^*})}} \|x^*\|.
$$
This proves assertion (i).
With minor changes, the above proof is applicable to  assertion (ii).
\qed\end{proof}

\begin{remark}
The above definitions of the limiting $g$-coderivatives follow the original idea of limiting coderivatives; cf. \cite{Mor06.1}.
In particular, they define positively homogeneous \SVM s with not necessarily convex graphs.
However, there are also several important distinctions.
Firstly, similar to the corresponding definition introduced in \cite{IofOut08}, these are ``outer'' objects: only sequences $(x_k,y_k)\in\gph F$ with $x_k$ components lying outside of the set $F\iv(\by)$ are taken into consideration.
Secondly, as it is reflected in their names, each of the limiting outer $g$-coderivatives depends on properties of the function $g$, more specifically on properties of its Fr\'echet subdifferentials near $\by$.
It is not excluded in any of the definitions that $\|v_k^*\|\to\infty$ and consequently $\|y_k^*\|\to\infty$, and nevertheless the sequence $(y_k^*)$ produces a finite element $y^*\in Y$.
\end{remark}

\begin{remark}\label{r23}
The definitions of the limiting $g$-coderivatives can be simplified if one imposes an additional requirement on $g$, namely that $\|v^*\|\ge\al$ for some $\al>0$ and all $v^*\in\sd g(y)$ when $y\in Y$ is sufficiently close to $\by$.
Then the case $y_k^*=0$ $(\forall k\in\N)$ can be dropped.
%However, such a requirement is too restrictive for the applications in mind.
\end{remark}

\begin{remark}
Analyzing the definitions of the limiting $g$-coderivatives and the proof of Proposition~\ref{ows}, one can notice that there is no need to care much about the convergence of the sequences in $Y^*$.
The limiting $g$-coderivatives in Proposition~\ref{ows} can be replaced by the corresponding limiting sets in $X^*$ only.
For example, instead of the limiting outer $g$-coderivative defined by \eqref{D*1}, one can consider the following simplified set:
\begin{align}\notag
{S}{}^{*>}_gF (\bx,\by):=& \{x^*\in X^*\mid
\exists (x_k,y_k,x^*_k,y^*_k,v^*_k)\subset X\times Y\times X^*\times Y^*\times Y^*
\\\notag
&\mbox{such that}\;
(x_k,y_k)\in\gph F,\;x_k\notin F\iv(\by),\;
(y_k^*,x_k^*)\in\gph{D}{}^{*}F(x_k,y_k),\; \\\notag
&v^*_k\in\sd g(y_k),\;
(x_k,y_k)\to(\bx,\by),\;y^*_k-v^*_k\to0,\; x^*_k\to x^*\}.
\end{align}
Proposition~\ref{ows} (i) remains true if $\overline{D}{}^{*>}_gF(\bx,\by)(\Sp^*_{Y^*})$ there is replaced by ${S}{}^{*>}_gF(\bx,\by)$.
This way, one can also relax the assumption that $\dim Y<\infty$.
\end{remark}

\begin{remark}\label{Rem9}
One can define also $g$-coderivative (indirect) counterparts of the modified strict subdifferential $g$-slopes \eqref{5mf03} and \eqref{5mf03a}.
It is sufficient to add to the list of properties in definitions \eqref{D*1} and \eqref{D*3} an additional requirement that $g(y_k)/\|x_k-\bx\|\to0$ as $k\to\infty$.
The corresponding sets can be used for characterizing metric $g$-sub\-regularity.
However, the analogues of the equalities in Proposition~\ref{ows} would not hold for them.
\end{remark}

\subsection{Criteria of metric $g$-subregularity}
The next theorem is a consequence of Theorem~\ref{3T1}.

\begin{theorem}\label{5T1}
\begin{enumerate}
\item
$\sr_g[F](\bx,\by)\le
\overline{|\nabla{F}|}{}^{\diamond}_g(\bar{x},\by)$;
\item
if $X$ and $Y$ are complete and $\gph F$ is locally closed
near $(\bar{x},\by)$,
then
$\sr_g[F](\bx,\by)=
\overline{|\nabla{F}|}{}^{\diamond}_g(\bar{x},\by)$.
\end{enumerate}
\end{theorem}

The next two corollaries summarize quantitative and qualitative criteria of metric $g$-sub\-regu\-larity, respectively.

\begin{corollary}[Quantitative criteria]\label{5C1.1}
Let $\ga>0$.
Consider the following conditions:
\renewcommand {\theenumi} {\alph{enumi}}
\begin{enumerate}
\item
$F$ is metrically $g$-subregular at $(\bx,\by)$ with some $\tau>0$;
\item
$\overline{|\nabla{F}|}{}_g^{\diamond}(\bar{x},\by)>\ga$,\\ i.e., for some $\rho>0$ and any $(x,y)\in\gph F$ with $x\notin F\iv(\by)$, $d(x,\bx)<\rho$, and $d(y,\by)<\rho$, it holds $|\nabla{F}|_{g,\rho}^{\diamond}(x,y)>\ga$, and consequently there is a $(u,v)\in\gph F$, such that
\begin{gather}\label{5ree}
g(y)-g(v)>\ga d_\rho((u,v),(x,y));
\end{gather}
\item
$\ds\liminf_{\substack{x\to\bx\\x\notin F\iv(\by),\,y\in F(x)}} \frac{g(y)}{d(x,\bx)}>\ga$;
\item
$\overline{|\nabla{F}|}{}_g(\bar{x},\by)>\ga$,\\ i.e., for some $\rho>0$ and any $(x,y)\in\gph F$ with $x\notin F\iv(\by)$, $d(x,\bx)<\rho$, and $d(y,\by)<\rho$, it holds $|\nabla{F}|_{g,\rho}(x,y)>\ga$, and consequently, for any $\eps>0$, there is a $(u,v)\in\gph F\cap B_\eps(x,y)$, such that \eqref{5ree} holds true;
\item
$\overline{|\nabla{F}|}{}^+_g(\bar{x},\by)>\ga$,\\ i.e., for some $\rho>0$ and any $(x,y)\in X\times Y$ with $x\notin F\iv(\by)$, $d(x,\bx)<\rho$, $d(y,\by)<\rho$, and $g(y)/d(x,\bx)\le\ga$, it holds $|\nabla{F}|_{g,\rho}(x,y)>\ga$ and consequently, for any $\eps>0$, there is a $(u,v)\in\gph F\cap B_\eps(x,y)$, such that \eqref{5ree} holds true;
\item
$X$ and $Y$ are normed spaces and $\overline{|\sd{F}|}{}_g^a(\bar{x},\by)>\ga$,\\ i.e.,
for some $\rho>0$ and any $(x,y)\in\gph F$ with $x\notin F\iv(\by)$, $\|x-\bx\|<\rho$, and $\|y-\by\|<\rho$, it holds $|\sd{F}|^a_{g,\rho}(x,y)>\ga$, and consequently there exists an $\eps>0$, such that
\begin{gather}\label{5f}
\|x^*\|>\ga\quad\mbox{for all } x^*\in D^*F(x,y)(\sd g(B_\eps(y))+\rho\B^*);
\end{gather}
\item
$X$ and $Y$ are normed spaces and $\overline{|\sd{F}|}{}_g^{a+}(\bar{x},\by)>\ga$,\\ i.e., for some $\rho>0$ and any $(x,y)\in X\times Y$ with $x\notin F\iv(\by)$, $\|x-\bx\|<\rho$, $\|y-\by\|<\rho$, and $g(y)/\|x-\bx\|\le\ga$, it holds $|\sd{F}|_{g,\rho}^a(x,y)>\ga$ and consequently, there exists an $\eps>0$, such that
\eqref{5f} holds true;
\item
$X$ and $Y$ are normed spaces and $\overline{|\sd{F}|}{}_g(\bar{x},\by)>\ga$,\\ i.e.,
for some $\rho>0$ and any $(x,y)\in\gph F$ with $x\notin F\iv(\by)$, $\|x-\bx\|<\rho$, and $\|y-\by\|<\rho$, it holds $|\sd{F}|_{g,\rho}(x,y)>\ga$, and consequently \begin{gather}\label{5f2}
\|x^*\|>\ga\quad\mbox{for all } x^*\in D^*F(x,y)(\sd g(y)+\rho\B^*);
\end{gather}
\sloppy
\item
$X$ and $Y$ are normed spaces and $\overline{|\sd{F}|}{}^+_g(\bar{x},\by)>\ga$,\\ i.e., for some $\rho>0$ and any $(x,y)\in X\times Y$ with $x\notin F\iv(\by)$, $\|x-\bx\|<\rho$, $\|y-\by\|<\rho$, and $g(y)/\|x-\bx\|\le\ga$, it holds $|\sd{F}|_{g,\rho}(x,y)>\ga$ and consequently, \eqref{5f2} holds true;
\item
$X$ and $Y$ are finite dimensional normed spaces and
\begin{gather*}
\|x^*\|>\ga\quad\mbox{for all } x^*\in \overline{D}{}^{*>a}_gF(\bx,\by) (\Sp^*_{Y^*});
\end{gather*}
\item
$X$ and $Y$ are finite dimensional normed spaces and
\begin{gather*}
\|x^*\|>\ga\quad\mbox{for all } x^*\in \overline{D}{}^{*>}_gF(\bx,\by) (\Sp^*_{Y^*}).
\end{gather*}
\end{enumerate}
\renewcommand {\theenumi} {\roman{enumi}}
The following implications hold true:
\begin{enumerate}
\item
{\rm (c) \folgt (e)},
{\rm (d) \folgt (e)},
{\rm (e) \folgt (b)},
{\rm (f) \folgt (g) \folgt (i)}, {\rm (f)~\folgt (h)~\folgt (i)}, {\rm (j)~\folgt (k)};
\item
if $\ga<\tau$, then {\rm (a) \folgt (b)};
\item
if $\tau\le\ga$, $X$ and $Y$ are complete, and $\gph F$ is locally closed near $(\bar{x},\by)$, then
{\rm (b) \folgt (a)}.
\cnta
\end{enumerate}
Suppose $X$ and $Y$ are normed spaces.
\begin{enumerate}
\cntb
\item
{\rm (f)~\folgt (d)} and {\rm (g)~\folgt (e)},
provided that $X$ and $Y$ are Asplund and $\gph F$ is locally closed near $(\bx,\by)$;
\item
{\rm (h)~\iff (d)} and {\rm (i)~\iff (e)},\\
provided that
$g$ is Fr\'echet differentiable near $\by$ except $\by$ and one of the following conditions is satisfied:
\begin{enumerate}
\item
$X$ and $Y$ are Asplund and $\gph F$ is locally closed near $(\bar{x},\by)$;
\item
$F$ is convex;
\end{enumerate}
\item
{\rm (b)~\iff (d)~\iff (e)~\iff (h)~\iff (i)}, provided that $F$ and $g$ are convex;
\item
{\rm (f)~\iff (j)} and {\rm (h)~\iff (k)}, provided that  $\dim X<\infty$ and $\dim Y<\infty$.
\end{enumerate}
\end{corollary}

The conclusions of Corollary~\ref{5C1.1} are illustrated in
Fig.~\ref{fig.5}.

\begin{figure}[!htb]
$$\xymatrix{
&&&&{\rm (c)}\ar[d]
\\
&&{\rm (d)}\ar[rr]
\ar @{} [drr] |{\substack{X,Y\,{\rm Asplund}\\\gph F\, {\rm closed}}}
&&{\rm (e)}\ar[rr]
\ar@/_/@{-->}[ll]_{F,g\,{\rm convex}}
\ar@/^1pc/@{-->}[dd]
&&{\rm (b)} \ar@/_/@{-->}[rr]_{\substack{\tau\le\ga\\X,Y\, {\rm complete}\\\gph F\, {\rm closed}}}
\ar@/_/@{-->}[ll]_{F,g\,{\rm convex}}
&&{\rm (a)} \ar@/_/@{-->}[ll]_{\ga<\tau}
\\
{\rm (j)}\ar@{-->}[rr]\ar[d]
\ar @{} [drr] |{\substack{\dim X<\infty\\\dim Y<\infty}}
&&{\rm (f)}\ar[rr]
\ar@{-->}[u]
\ar[d]
\ar@{-->}[ll]
&&{\rm (g)}
\ar@{-->}[u]
\ar[d]
\\
{\rm (k)}\ar@{-->}[rr]
&&{\rm (h)}\ar[rr]\ar@{-->}[ll]
&&{\rm (i)}
\ar@/_1pc/@{-->}[uu]_{\substack{[g\,{\rm differentiable}\smallskip\\(X,Y\, {\rm Asplund}\\\gph F\, {\rm closed})\\{\rm or}\\ F\,{\rm convex}]\\{\rm or}\smallskip\\ F,g\,{\rm convex}}}
\ar@/^/@{-->}[ll]^{F,g\,{\rm convex}}
}$$
\caption{Corollary~\ref{5C1.1} \label{fig.5}}
\end{figure}

\begin{remark}
The existence of a $\ga>0$ such that one of the conditions (j) or (k) in Corollary~\ref{5C1.1} holds true is equivalent to the kernel of the corresponding limiting outer $g$-co\-derivative being equal to $\{0\}$, which is a traditional type of a qualitative coderivative regularity condition.
Conditions (j) and (k), on the other hand, provide additionally quantitative estimates of the regularity modulus.
\end{remark}

\begin{corollary}[Qualitative criteria]\label{5C1.2}
Suppose $X$ and $Y$ are complete metric spaces and $\gph F$ is locally closed near $(\bar{x},\by)$.
Then,
$F$ is metrically $g$-subregular at $(\bx,\by)$ if one of the following conditions holds true:
\renewcommand {\theenumi} {\alph{enumi}}
\begin{enumerate}
\item
$\overline{|\nabla{F}|}{}^{\diamond}_g(\bar{x},\by)>0$; \item
$\ds\liminf_{\substack{x\to\bx\\x\notin F\iv(\by),\,y\in F(x)}} \frac{g(y)}{d(x,\bx)}>0$;
\item
$\overline{|\nabla{F}|}{}_g(\bar{x},\by)>0$, or equivalently,
$\ds\lim_{\rho\downarrow0}
\inf_{\substack{d(x,\bx)<\rho,\,d(y,\by)<\rho\\
(x,y)\in\gph F,\,x\notin F\iv(\by)}}\,
|\nabla{F}|_{g,\rho}(x,y)>0;$
\item
$\overline{|\nabla{F}|}{}^+_g(\bar{x},\by)>0$, or equivalently,
$\ds\lim_{\rho\downarrow0}
\inf_{\substack{d(x,\bx)<\rho,\,\frac{g(y)}{d(x,\bx)}<\rho\\
(x,y)\in\gph F,\,x\notin F\iv(\by)}}\,
|\nabla{F}|_{g,\rho}(x,y)>0.$
\cnta
\end{enumerate}
If $X$ and $Y$ are Asplund spaces, then the following conditions are also sufficient:
\begin{enumerate}
\cntb
\item
$\overline{|\sd{F}|}{}^a_g(\bar{x},\by)>0$, or equivalently,
$\ds\lim_{\rho\downarrow0}
\inf_{\substack{\|x-\bx\|<\rho,\,\|y-\by\|<\rho\\
(x,y)\in\gph F,\,x\notin F\iv(\by)}}\,
|\sd{F}|_{g,\rho}^a(x,y)>0;$
\item
$\overline{|\sd{F}|}{}^{a+}_g(\bar{x},\by)>0$, or equivalently,
$\ds\lim_{\rho\downarrow0}
\inf_{\substack{\|x-\bx\|<\rho,\,\frac{g(y)}{\|x-\bx\|}<\rho\\
(x,y)\in\gph F,\,x\notin F\iv(\by)}}\,
|\sd{F}|_{g,\rho}^a(x,y)>0.$
\cnta
\end{enumerate}
If $X$ and $Y$ are Banach spaces, then the next two conditions:
\begin{enumerate}
\cntb
\item
$\overline{|\sd{F}|}{}_g(\bar{x},\by)>0$, or equivalently,
$
\ds\lim_{\rho\downarrow0}
\inf_{\substack{\|x-\bx\|<\rho,\,\|y-\by\|<\rho\\
(x,y)\in\gph F,\,x\notin F\iv(\by)}}\,
|\sd{F}|_{g,\rho}(x,y)>0,
$
\item
$\overline{|\sd{F}|}{}^+_g(\bar{x},\by)>0$, or equivalently,
$\ds\lim_{\rho\downarrow0}
\inf_{\substack{\|x-\bx\|<\rho,\,\frac{g(y)}{\|x-\bx\|}<\rho\\
(x,y)\in\gph F,\,x\notin F\iv(\by)}}\,
|\sd{F}|_{g,\rho}(x,y)>0,$
\cnta
\end{enumerate}
are sufficient, provided that one of the following conditions is satisfied:
\begin{itemize}
\item
$X$ and $Y$ are Asplund spaces and $g$ is Fr\'echet differentiable near $\by$ except $\by$,
\item
$F$ is convex and $g$ is either convex or Fr\'echet differentiable near $\by$ except $\by$.
\end{itemize}
If $X$ and $Y$ are finite dimensional normed spaces, then the following conditions are also sufficient:
\begin{enumerate}
\cntb
\item
$0\notin \overline{D}{}^{*>a}_gF(\bx,\by) (\Sp^*_{Y^*})$;
\item
$0\notin \overline{D}{}^{*>}_gF(\bx,\by) (\Sp^*_{Y^*})$,
provided that $F$ is convex, and $g$ is either convex or Fr\'echet differentiable near $\by$ except $\by$.
\end{enumerate}
\renewcommand {\theenumi} {\roman{enumi}}
Moreover,
\begin{enumerate}
\item
condition {\rm (a)} is also necessary for the metric $g$-subregularity of $F$ at $(\bx,\by)$;
\item
{\rm (b) \folgt (d)},
{\rm (c) \folgt (d)},
{\rm (d) \folgt (a)},
{\rm (e) \folgt (f) \folgt (h)},
{\rm (e) \folgt (g) \folgt (h)},
{\rm (i) \folgt (j)}.
\cnta
\end{enumerate}
Suppose $X$ and $Y$ are Banach spaces.
\begin{enumerate}
\cntb
\item
If $X$ and $Y$ are Asplund, then {\rm (e) \folgt (c)} and {\rm (f) \folgt (d)};
\item
if $g$ is Fr\'echet differentiable near $\by$ except $\by$ and either $X$ and $Y$ are Asplund or $F$ is convex, then {\rm (e) \iff (c)} and {\rm (f) \iff (d)};
\item
if $F$ and $g$ are convex, then {\rm (a) \iff (c) \iff (d) \iff \rm (g) \iff (h)};
\item
if $X$ and $Y$ are finite dimensional normed spaces, then
{\rm (e) \iff (i)} and
{\rm (g) \iff (j)}.
\end{enumerate}
\end{corollary}

The conclusions of Corollary~\ref{5C1.2} are illustrated in
Fig.~\ref{fig.6}.

\begin{figure}[!htb]
$$\xymatrix@C=1cm{
&&*+[F]{\sr_g[F](\bx,\by)>0} \ar[d]
\\
&*+[F]{\ds\liminf_{\substack{x\to\bx\\x\notin F\iv(\by)\\ y\in F(x)}} \frac{g(y)}{d(x,\bx)}>0}\ar[dr]
&*+[F]{\overline{|\nabla{F}|}{}^{\diamond}_g(\bar{x},\by)>0}
\ar[u]
\ar@/^/@{-->}[d]^{F,g\,{\rm convex}}
\\
&*+[F]{\overline{|\nabla{F}|}_g(\bar{x},\by)>0}
\ar[r]
\ar@/^/@{-->}[d]
\ar@{} [dr] |{\substack{X,Y\,{\rm Banach},\,g\,{\rm differentiable}\smallskip\\X,Y\,{\rm Asplund}\,{\rm or}\,F\,{\rm convex}}}
&*+[F]{\overline{|\nabla{F}|}{}^{+}_g(\bar{x},\by)>0}
\ar[u]
\ar@/_/@{-->}[d]
\ar@/_/@{-->}[l]_{F,g\,{\rm convex}}
\\
*+[F]{0\notin \overline{D}{}^{*>a}_gF(\bx,\by) (\Sp^*_{Y^*})}
\ar[d]
\ar@{-->}[r]
\ar@{} [dr] |{\qquad\substack{\dim X<\infty\\\dim Y<\infty}} &*+[F]{\overline{|\sd{F}|}{}^a_g(\bar{x},\by)>0}
\ar[r]
\ar@{-->}[l]
\ar@/^/@{-->}[u]^{X,Y\,{\rm Asplund}}
\ar[d]
&*+[F]{\overline{|\sd{F}|}{}^{a+}_g(\bar{x},\by)>0}
\ar@/_/@{-->}[u]_{X,Y\,{\rm Asplund}}
\ar[d]
\\
*+[F]{0\notin \overline{D}{}^{*>}_gF(\bx,\by) (\Sp^*_{Y^*})}
\ar@{-->}[r]
&*+[F]{\overline{|\sd{F}|}_g(\bar{x},\by)>0}
\ar[r]
\ar@{-->}[l]
&*+[F]{\overline{|\sd{F}|}{}^{+}_g(\bar{x},\by)>0}
\ar@/_4.5pc/@{-->}[uu]|{\substack{X,Y\,{\rm Banach}\\F,g\,{\rm convex}}}
\ar@/^/@{-->}[l]^{F,g\,{\rm convex}}
}$$
\caption{Corollary~\ref{5C1.2}}\label{fig.6}
\end{figure}

\section{Metric $\varphi$-subregularity}\label{S50}

\subsection{Definition}
In this section, for a set-valued mapping $F:X\rightrightarrows Y$, we consider the property of metric $\varphi$-subregu\-larity being a realization of the property of metric $g$-subregularity in the case when $g$ has a special structure:
\begin{gather}\label{gphi}
g(y)=\varphi(d(y,\by)),\quad y\in Y,
\end{gather}
where $\varphi:\R_+\to\R_+$ is continuously differentiable, with (possibly infinite) $\varphi'(0)$ understood as the right-hand derivative, and satisfies the following properties:
\begin{itemize}
\item [($\Phi1$)]
$\varphi(0)=0$,
\item [($\Phi2$)]
$\varphi'(t)>0$ for all $t\in\R_+$.
\end{itemize}

Thanks to ($\Phi2$), $\varphi$ is an increasing function.
Hence, $\varphi(t)>0$ for all $t>0$.
Obviously, function $g$ defined by \eqref{gphi} is continuous at $\by$, locally Lipschitz continuous on $Y\setminus\{\by\}$ and satisfies $g(\by)=0$ and properties (P1$'$) and (P2$'$).

\begin{remark}\label{R5.1}
The requirement of continuous differentiability of $\varphi$ and property $(\Phi2)$ can be weakened.
For many estimates, it is sufficient to assume that
$\varphi$ is differentiable on $(0,\de)$ for some $\de>0$ and
$\liminf_{t\downarrow0}\varphi'(t)>0$.
\end{remark}

We say that $F$ is metrically \emph{$\varphi$-subregu\-lar} at $(\bx,\by)$ with constant $\tau>0$ iff there exists a neighbourhood $U$ of $\bx$, such that
\begin{equation}\label{MR3}
\tau d(x,F^{-1}(\by))\le \varphi(d(y,\by))\quad \mbox{for all } x\in U,\; y\in F(x),
\end{equation}
or, taking into account the monotonicity of $\varphi$,
\begin{equation}\label{MR4}
\tau d(x,F^{-1}(\by))\le \varphi(d(\by,F(x)))\quad \mbox{for all } x\in U.
\end{equation}

Metric $\varphi$-subregularity can be equivalently characterized using the following constant being the realization of \eqref{5CMR-g}:
\begin{gather}\label{6CMR-g}
\sr_\varphi[F](\bx,\by):=
\liminf_{\substack{x\to\bx\\x\notin F\iv(\by)}} \frac{\varphi(d(\by,F(x)))}{d(x,F^{-1}(\by))}.
\end{gather}

If $\varphi$ is the identity function, i.e., $\varphi(t)=t$ for all $t\in\R_+$, then \eqref{MR4} (and \eqref{MR3}) reduces to the standard definition of metric subregularity; cf. \cite{Kru15}.
Model \eqref{MR4} covers also more general, nonlinear regularity properties.
For instance, if $\varphi(t)=t^q$, $t\in\R_+$, with $0<q\le1$, then \eqref{MR4} turns into the definition of H\"older metric subregularity; cf. \cite{Kru15.2}.

\subsection{Primal space and subdifferential slopes}

Given a $\rho>0$ and $(x,y)\in\gph F$, as the main primal space local tool, in this section we are going to use the \emph{$\rho$-slope} of $F$ at $(x,y)$:
\begin{gather}\label{srho}
|\nabla{F}|_{\rho}(x,y):=
\limsup_{\substack{(u,v)\to(x,y),\,
(u,v)\ne(x,y)\\
(u,v)\in\gph F}}
\frac{[d(y,\by)-d(v,\by)]_+} {d_\rho((u,v),(x,y))},
\end{gather}
while the definition
\begin{gather}\label{6nls}
|\nabla{F}|_{\varphi,\rho}^{\diamond}(x,y):=
\sup_{\substack{(u,v)\ne(x,y)\\(u,v)\in\gph F}}
\frac{\varphi(d(y,\by))-\varphi(d(v,\by))}{d_\rho((u,v),(x,y))},
\end{gather}
of the \emph{nonlocal $(\varphi,\rho)$-slope} of $F$ at $(x,y)$ involves $\varphi$ and is the realization of the nonlocal $g$-slope \eqref{5nls} for the case when $g$ is given by \eqref{gphi}.

It is easy to notice, that the $\rho$-slope \eqref{srho} is the realization of the $g$-slope \eqref{5ls} for the case when $g(y)=d(y,\by)$, $y\in Y$.
When $g$ is given by \eqref{gphi}, the $g$-slope \eqref{5ls} still has a simple representation in terms of \eqref{srho}.

\begin{proposition}\label{6P0}
Suppose $(x,y)\in\gph F$, $y\ne\by$, $\rho>0$, and $g$ is given by \eqref{gphi}.
Then,
$$|\nabla{F}|_{g,\rho}(x,y)
=\varphi'(d(y,\by))\;|\nabla{F}|_{\rho}(x,y).$$
\end{proposition}

\begin{proof}
By \eqref{5ls}, the differentiability of $\varphi$, ($\Phi2$), and \eqref{srho},
\begin{align*}
|\nabla{F}|_{g,\rho}(x,y)
&=\limsup_{\substack{(u,v)\to(x,y),\,
(u,v)\ne(x,y)\\(u,v)\in\gph F}}
\frac{[\varphi(d(y,\by))-\varphi(d(v,\by))]_+} {d_\rho((u,v),(x,y))}
\\
&=\limsup_{\substack{(u,v)\to(x,y),\,
(u,v)\ne(x,y)\\(u,v)\in\gph F}}
\frac{[\varphi'(d(y,\by))(d(y,\by)-d(v,\by))+o(d(v,y))]_+} {d_\rho((u,v),(x,y))}
\\
&=\varphi'(d(y,\by)) \limsup_{\substack{(u,v)\to(x,y),\,
(u,v)\ne(x,y)\\(u,v)\in\gph F}}
\frac{[d(y,\by)-d(v,\by)]_+} {d_\rho((u,v),(x,y))}
=\varphi'(d(y,\by))\;|\nabla{F}|_{\rho}(x,y).
\end{align*}
In the above formula, $o(\cdot)$ stands for a function from $\R_+$ to $\R_+$ with the property $o(t)/t\to0$ as $t\downarrow0$.
\qed\end{proof}

Thanks to Proposition~\ref{6P0}, the strict $g$-slopes \eqref{5ss}--\eqref{5uss} produce the following definitions:
\begin{gather}\label{phi-ss}
\overline{|\nabla{F}|}{}_\varphi(\bar{x},\by):=
\lim_{\rho\downarrow0}
\inf_{\substack{d(x,\bx)<\rho,\,d(y,\by)<\rho\\
(x,y)\in\gph F,\,x\notin F\iv(\by)}}\,
\varphi'(d(y,\by))\;|\nabla{F}|_{\rho}(x,y),
\\\label{mphi-ss}
\overline{|\nabla{F}|}{}_\varphi^{+}(\bar{x},\by):=
\lim_{\rho\downarrow0}
\inf_{\substack{d(x,\bx)<\rho,\,d(y,\by)<\rho\\
(x,y)\in\gph F,\,x\notin F\iv(\by)}}
\max\left\{\varphi'(d(y,\by))|\nabla{F}|_{\rho}(x,y), \frac{\varphi(d(y,\by))}{d(x,\bx)}\right\},
\\\label{6uss}
\overline{|\nabla{F}|}{}^{\diamond}_\varphi(\bar{x},\by):=
\lim_{\rho\downarrow0}
\inf_{\substack{d(x,\bx)<\rho,\,d(y,\by)<\rho\\
(x,y)\in\gph F,\,x\notin F\iv(\by)}}\,
|\nabla{F}|{}^{\diamond}_{\varphi,\rho}(x,y).
\end{gather}
We call the above constants, respectively,
the \emph{strict $\varphi$-slope}, \emph{modified strict $\varphi$-slope}, and \emph{uniform strict $\varphi$-slope} of $F$ at $(\bx,\by)$.

%\subsection{Subdifferential slopes}
If $X$ and $Y$ are normed spaces, we define the \emph{subdifferential $\rho$-slope} and \emph{approximate subdifferential $\rho$-slope} ($\rho>0$) of $F$ at $(x,y)\in\gph F$ with $y\ne\by$ as follows:
\begin{gather}\label{6srs}
|\sd{F}|_{\rho}(x,y)
:=\inf_{\substack{x^*\in D^*F(x,y)(J(y-\by)+\rho\B^*)}}
\|x^*\|,
\\\label{6asrs}
|\sd{F}|^a_{\rho}(x,y)
:=\liminf_{\substack{v\to y-\by}}\
\inf_{\substack{x^*\in D^*F(x,y)(J(v)+\rho\B^*)}}
\|x^*\|.
\end{gather}
$J$ in the above formulas stands for the duality mapping \eqref{J}.

Similar to \eqref{srho}, constants \eqref{6srs} and \eqref{6asrs} are the realizations of the $(g,\rho)$-slopes \eqref{5f02} and \eqref{5f02a}, respectively, in the case $g(y)=\|y-\by\|$.
They do not depend on $\varphi$.
Using some simple calculus, one can formulate representations for $(g,\rho)$-slopes \eqref{5f02} and \eqref{5f02a} in the case when $g$ is given by \eqref{gphi}.
In the next proposition and the rest of the article, we use the notation
$$\xi_\varphi(y):=(\varphi'(\|y-\by\|))\iv.$$

\begin{proposition}\label{6P1}
Suppose $X$ and $Y$ are normed spaces, $(x,y)\in\gph F$, $y\ne\by$, $\rho>0$, and $g$ is given by \eqref{gphi}.
The following representations hold:
\begin{enumerate}
\item
$\sd g(y)=\varphi'(\|y-\by\|)J(y-\by)$;
\item
$|\sd{F}|_{g,\rho}(x,y) =\varphi'(\|y-\by\|)\;|\sd{F}|_{\xi_\varphi(y)\rho}(x,y)$;
\item
$|\sd{F}|^a_{g,\rho}(x,y) =\varphi'(\|y-\by\|)\;|\sd{F}|^a_{\xi_\varphi(y)\rho}(x,y)$.
\end{enumerate}
\end{proposition}

\begin{proof}
(i) follows from the composition rule for Fr\'echet subdifferentials (see, e.g., \cite[Corollary 1.14.1]{Kru03.1}).

(ii) Substituting (i) into \eqref{5f02} and taking into account that $D^*F(x,y)(tv^*)=tD^*F(x,y)(v^*)$ for any $v^*\in Y^*$ and $t>0$, we obtain:
\begin{align*}
|\sd{F}|_{g,\rho}(x,y)
&=\inf_{\substack{x^*\in D^*F(x,y)(\varphi'(\|y-\by\|)J(y-\by)+\rho\B^*)}}
\|x^*\|
\\
&=\varphi'(\|y-\by\|)\inf_{\substack{x^*\in D^*F(x,y)(J(y-\by)+\xi_\varphi(y)\rho\B^*)}}
\|x^*\|
=\varphi'(\|y-\by\|)\;|\sd{F}|_{\xi_\varphi(y)\rho}(x,y).
\end{align*}
Similarly, substituting (i) into \eqref{5f02a}, we obtain (iii).
\qed\end{proof}

%\subsection{Strict $\varphi$-slopes}

Now we can define the \emph{strict subdifferential $\varphi$-slope}, \emph{approximate strict subdifferential $\varphi$-slope}, \emph{modified strict subdifferential $\varphi$-slope}, and \emph{modified approximate strict subdifferential $\varphi$-slope} of $F$ at $(\bx,\by)$:
\begin{gather}\label{phi-sss}
\overline{|\sd{F}|}{}_\varphi(\bar{x},\by):=
\lim_{\rho\downarrow0}
\inf_{\substack{\|x-\bx\|<\rho,\,\|y-\by\|<\rho\\
(x,y)\in\gph F,\,x\notin F\iv(\by)}}\,
\varphi'(\|y-\by\|)\;|\sd{F}|_{\xi_\varphi(y)\rho}(x,y),
\\\label{phi-asss}
\overline{|\sd{F}|}{}^a_\varphi(\bar{x},\by):=
\lim_{\rho\downarrow0}
\inf_{\substack{\|x-\bx\|<\rho,\,\|y-\by\|<\rho\\
(x,y)\in\gph F,\,x\notin F\iv(\by)}}\,
\varphi'(\|y-\by\|)\;|\sd{F}|^a_{\xi_\varphi(y)\rho}(x,y),
\\\label{phi-msss}
\overline{|\sd{F}|}{}_\varphi^{+}(\bar{x},\by):=
\lim_{\rho\downarrow0}
\inf_{\substack{\|x-\bx\|<\rho,\,\|y-\by\|<\rho\\
(x,y)\in\gph F,\,
x\notin F\iv(\by)}}
\max\left\{\varphi'(\|y-\by\|)|\sd{F}|_{\xi_\varphi(y)\rho}(x,y), \frac{\varphi(\|y-\bar y\|}{\|x-\bx\|}\right\},
\\\label{phi-masss}
\overline{|\sd{F}|}{}_\varphi^{a+}(\bar{x},\by):=
\lim_{\rho\downarrow0}
\inf_{\substack{\|x-\bx\|<\rho,\,\|y-\by\|<\rho\\
(x,y)\in\gph F,\,
x\notin F\iv(\by)}}
\max\left\{\varphi'(\|y-\by\|)|\sd{F}|^a_{\xi_\varphi(y)\rho}(x,y), \frac{\varphi(\|y-\bar y\|}{\|x-\bx\|}\right\}.
\end{gather}
In view of Proposition~\ref{6P1}, these constants coincide, respectively, with the corresponding strict subdifferential $g$-slopes \eqref{5f03}, \eqref{5f03a}, \eqref{5mf03}, and \eqref{5mf03a}.
Factor $\xi_\varphi(y)$ in \eqref{phi-sss}--\eqref{phi-masss} cannot be dropped in general.

\begin{proposition}\label{6P2}
Suppose $X$ and $Y$ are normed spaces.
The following assertions hold true:
\begin{enumerate}
\item
$\ds\overline{|\sd{F}|}{}_\varphi(\bar{x},\by)\ge
\lim_{\rho\downarrow0}
\inf_{\substack{\|x-\bx\|<\rho,\,\|y-\by\|<\rho\\
(x,y)\in\gph F,\,x\notin F\iv(\by)}}\,
\varphi'(\|y-\by\|)\;|\sd{F}|_{\rho}(x,y)$;
\item
$\ds\overline{|\sd{F}|}{}^a_\varphi(\bar{x},\by)\ge
\lim_{\rho\downarrow0}
\inf_{\substack{\|x-\bx\|<\rho,\,\|y-\by\|<\rho\\
(x,y)\in\gph F,\,x\notin F\iv(\by)}}\,
\varphi'(\|y-\by\|)\;|\sd{F}|^a_{\rho}(x,y)$;
\item
$\ds\overline{|\sd{F}|}{}_\varphi^{+}(\bar{x},\by)\ge
\lim_{\rho\downarrow0}
\inf_{\substack{\|x-\bx\|<\rho,\,\|y-\by\|<\rho\\
(x,y)\in\gph F,\,
x\notin F\iv(\by)}}
\max\left\{\varphi'(\|y-\by\|)|\sd{F}|_{\rho}(x,y), \frac{\varphi(\|y-\bar y\|)}{\|x-\bx\|}\right\}$;
\item
$\ds\overline{|\sd{F}|}{}_\varphi^{a+}(\bar{x},\by)\ge
\lim_{\rho\downarrow0}
\inf_{\substack{\|x-\bx\|<\rho,\,\|y-\by\|<\rho\\
(x,y)\in\gph F,\,
x\notin F\iv(\by)}}
\max\left\{\varphi'(\|y-\by\|)|\sd{F}|^a_{\rho}(x,y), \frac{\varphi(\|y-\bar y\|)}{\|x-\bx\|}\right\}$.
\end{enumerate}
If $\varphi'(0)<\infty$, then the above relations hold as equalities.
\end{proposition}

\begin{proof}
We consider the first inequality.
The others can be treated in the same way.
If $\overline{|\sd{F}|}{}_\varphi(\bar{x},\by)=\infty$, the inequality holds trivially.
Let $\overline{|\sd{F}|}{}_\varphi(\bar{x},\by)<\ga<\infty$.
Fix an arbitrary $\rho>0$ and choose an $\al>0$ and a $\rho'\in(0,\rho)$, such that $\varphi'(t)>\al$ for all $t\in[0,\rho')$ and $\rho'<\al\rho$.
By \eqref{phi-sss}, there exists an $(x,y)\in\gph F$ with $\|x-\bx\|<\rho'$, $\|y-\by\|<\rho'$ and $x\notin F\iv(\by)$; a $y^*\in Y^*$, an $x^*\in D^*F(x,y)(y^*)$, and a $v^*\in J(y-\by)$, such that $\|v^*-y^*\|\le(\varphi'(\|y-\by\|))\iv\rho'$ and $\varphi'(\|y-\by\|)\|x^*\|<\ga$.
Hence, $\|x-\bx\|<\rho$, $\|y-\by\|<\rho$, and $\|v^*-y^*\|\le\al\iv\rho'<\rho$, and consequently the \RHS\ of (i) is less than $\ga$.
The conclusion follows since $\ga$ was chosen arbitrarily.

Let $\varphi'(0)<\infty$.
To prove the opposite inequality, we proceed in the same way starting with the \RHS\ of (i).
If it is infinite, the opposite inequality holds trivially.
Suppose that the \RHS\ of (i) is less than some positive number $\ga$.
Fix an arbitrary $\rho>0$ and choose an $\al>0$ and a $\rho'\in(0,\rho)$, such that $\varphi'(t)<\al$ for all $t\in[0,\rho')$ (in view of continuity of $\varphi'$) and $\rho'<\al\iv\rho$.
For this $\rho'$, there exists an $(x,y)\in\gph F$ with $\|x-\bx\|<\rho'$, $\|y-\by\|<\rho'$ and $x\notin F\iv(\by)$; a $y^*\in Y^*$, an $x^*\in D^*F(x,y)(y^*)$, and a $v^*\in J(y-\by)$, such that $\|v^*-y^*\|\le\rho'$ and $\varphi'(\|y-\by\|)\|x^*\|<\ga$.
Hence, $\|x-\bx\|<\rho$, $\|y-\by\|<\rho$, and $\|v^*-y^*\|<\al\iv\rho<(\varphi'(\|y-\by\|))\iv\rho$, and consequently, by \eqref{phi-sss}, $\overline{|\sd{F}|}{}_\varphi(\bar{x},\by)<\ga$.
The conclusion follows since $\ga$ was chosen arbitrarily.
\qed\end{proof}

The next statement summarizes the relationships between the $\varphi$-slopes.
It is a consequence of Propositions~\ref{5P2},  \ref{5P3}, \ref{5P6}, and \ref{6P1}.

\begin{proposition}[Relationships between slopes]\label{6P3}
\begin{enumerate}
\item
$\ds|\nabla{F}|_{\varphi,\rho}^{\diamond}(x,y)\ge
\max\left\{\varphi'(d(y,\by))\;|\nabla{F}|_{\rho}(x,y), \frac{\varphi(d(y,\by))} {d_\rho((x,y),(\bx,\by))}\right\}$
for all $\rho>0$ and $(x,y)\in\gph F$ with $y\ne\by$;
\item
$\ds\overline{|\nabla{F}|}{}_\varphi^{\diamond}(\bar{x},\by)\ge
\overline{|\nabla{F}|}{}^+_{\varphi}(\bx,\by)\ge
\overline{|\nabla{F}|}{}_{\varphi}(\bx,\by)$.
\cnta
\end{enumerate}
%If $F$ and $\varphi$ are convex, then {\rm (i)} and {\rm (ii)} hold as equalities.\\
Suppose $X$ and $Y$ are normed spaces.
\begin{enumerate}
\cntb
\item
$|\sd{F}|^a_{\rho}(x,y) \le|\sd{F}|_{\rho}(x,y)$ for all $\rho>0$ and $(x,y)\in\gph F$;
\item
$\overline{|\sd{F}|}{}^a_\varphi(\bar{x},\by) \le\overline{|\sd{F}|}{}_\varphi(\bar{x},\by)\le \overline{|\sd{F}|}{}^+_\varphi(\bar{x},\by)$ and\\
$\overline{|\sd{F}|}{}^a_\varphi(\bar{x},\by) \le\overline{|\sd{F}|}{}^{a+}_\varphi(\bar{x},\by) \le\overline{|\sd{F}|}{}^+_\varphi(\bar{x},\by)$;
\item
$\overline{|\nabla{F}|}{}_\varphi(\bar{x},\by)
\ge
\overline{|\sd{F}|}{}^{a}_\varphi(\bar{x},\by)$ and $\overline{|\nabla{F}|}{}^+_\varphi(\bar{x},\by)
\ge
\overline{|\sd{F}|}{}^{a+}_\varphi(\bar{x},\by)$,\\
provided that $X$ and $Y$ are Asplund and $\gph F$ is locally closed near $(\bar{x},\by)$;
\item
$\ds\overline{|\nabla{F}|}{}_\varphi(\bar{x},\by) =\overline{|\sd{F}|}{}_\varphi(\bar{x},\by)$ and $\ds\overline{|\nabla{F}|}{}^+_\varphi(\bar{x},\by) =\overline{|\sd{F}|}{}^+_\varphi(\bar{x},\by)$,\\
provided that $Y$ is Fr\'echet smooth and one of the following conditions is satisfied:
\begin{enumerate}
\item
$X$ is Asplund and $\gph F$ is locally closed near $(\bar{x},\by)$;
\item
$F$ is convex;
if both $F$ and $\varphi$ are convex, then {\rm (i)} and {\rm (ii)} also hold as equalities.
\end{enumerate}
\end{enumerate}
\end{proposition}

\subsection{Limiting outer $\varphi$-coderivative}
In finite dimensions, one can define the
\emph{limiting outer $\varphi$-coderivative} of $F$ being the realization of the limiting $g$-coderivative \eqref{D*1} and a counterpart of the strict subdifferential $\varphi$-slopes \eqref{phi-sss} and \eqref{phi-asss}.
Note that, due to the assumptions imposed on $\varphi$, the definition takes a simpler form, cf. Remark~\ref{r23}.

\begin{align}\notag
\gph\overline{D}{}^{*>}_\varphi F (\bx,\by):=& \Big\{(y^*,x^*)\in Y^*\times X^*\mid
\exists (x_k,y_k,x^*_k,y^*_k,v^*_k)\subset X\times Y\times X^*\times Y^*\times Y^*\;\mbox{such that}
\\\notag
&(x_k,y_k)\in\gph F,\;x_k\notin F\iv(\by),
\;(y_k^*,x_k^*)\in\gph{D}{}^{*}F(x_k,y_k),\;
v^*_k\in J(y_k-\by),
\\\notag
&(x_k,y_k)\to(\bx,\by),\;y^*_k-\varphi'(\|y_k-\by\|)v^*_k\to0,\; \|y^*\|x^*_k\to x^*,
\\\label{D*21}
&\mbox{if}\; y^*\ne0,\; \mbox{then}\; \frac{y^*_k}{\|y_k^*\|}\to\frac{y^*}{\|y^*\|}\Big\}.
\end{align}

The above formula takes into account the representation from Proposition~\ref{6P1}(ii).
Thanks to the continuous differentiability of $\varphi$ and convexity of a norm, $\overline{\sd}g(y)=\sd g(y)$ for all $y\in Y\setminus\{0\}$.
Taking into consideration the closedness of the Fr\'echet normal cone, one can conclude that the realization of the approximate limiting outer $g$-coderivative \eqref{D*3}
also reduces to \eqref{D*21}.

\begin{remark}
One can define also a $\varphi$-coderivative counterpart of the modified strict subdifferential $\varphi$-slopes \eqref{phi-msss} and \eqref{phi-masss}; cf. Remark~\ref{Rem9}.
\end{remark}

The next proposition is a consequence of Proposition~\ref{ows}.

\begin{proposition}\label{ows2}
Suppose $X$ and $Y$ are finite dimensional normed linear spaces.
Then,
$$\overline{|\sd{F}|}{}_\varphi(\bar{x},\by) =\overline{|\sd{F}|}{}^{a}_\varphi(\bar{x},\by) =\inf\limits_ {\substack{x^*\in\overline{D}{}^{*>}_\varphi F(\bx,\by) (\Sp^*_{Y^*})}} \|x^*\|.$$
\end{proposition}

\subsection{Criteria of metric $\varphi$-subregularity}
The next theorem is a consequence of Theorem~\ref{5T1}.

\begin{theorem}\label{6T1}
\begin{enumerate}
\item
$\sr_\varphi[F](\bx,\by)\le
\overline{|\nabla{F}|}{}^{\diamond}_\varphi(\bar{x},\by)$;
\item
if $X$ and $Y$ are complete and $\gph F$ is locally closed near $(\bar{x},\by)$,
then
$\sr_\varphi[F](\bx,\by)=
\overline{|\nabla{F}|}{}^{\diamond}_\varphi(\bar{x},\by)$.
\end{enumerate}
\end{theorem}

The estimate in the next proposition can be useful when formulating necessary conditions of $\varphi$-subregu\-larity.
It incorporates the following constant characterizing the behaviour of $\varphi$ near $0$:
\begin{gather}\label{vth}
\vartheta[\varphi]:=\liminf_{t\downarrow0} \frac{t\varphi'(t)}{\varphi(t)}.
\end{gather}
It is well defined since $\varphi(t)>0$ for all $t>0$.
Obviously, $\vartheta[\varphi]\ge0$.
If $\varphi(t)=t^q$ ($t\ge0$), then $\vartheta[\varphi]=q$.
%Another important example will be considered in the next section.

\begin{proposition}\label{iti}
Suppose $X$ and $Y$ are normed spaces and $F$ is convex near $(\bx,\by)$.
Then, $\vartheta[\varphi]\;\sr_\varphi[F](\bar{x},\by) \le\overline{|\sd{F}|}_\varphi(\bar{x},\by)$.
\end{proposition}

\begin{proof}
If $\sr_\varphi[F](\bar{x},\by)=0$ or $\vartheta[\varphi]=0$, the conclusion is trivial.
Suppose $0<\tau<\sr_\varphi[F](\bar{x},\by)$ and $0<\ga_1<\ga_2<\vartheta[\varphi]$.
Then there exists a $\rho>0$, such that
\begin{gather}\label{iti1}
\tau d(x,F^{-1}(\by))<\varphi(\|y-\by\|), \quad \forall x\in B_\rho(\bx)\setminus F\iv(\by),\; y\in F(x),
\\\label{iti2}
\varphi'(\|y-\by\|)\|y-\by\|\ge\ga_2\varphi(\|y-\by\|), \quad \forall y\in B_\rho(\by).
\end{gather}
Since $\varphi'(0)>0$, we can also assume that \begin{gather}\label{iti3}
\rho<(\ga_2-\ga_1)\frac{\varphi(t)}{t}, \quad \forall t\in(0,\rho).
\end{gather}
Choose an arbitrary $(x,y)\in\gph F$ with $\|x-\bx\|<\rho$, $\|y-\by\|<\rho$, $x\notin F\iv(\by)$; $v^*\in J(y-\by)$; and $x^*\in D^*F(x,y)(v^*+\xi_\varphi(y)\rho\B^*)$ where $\xi_\varphi(y)=(\varphi'(\|y-\by\|))\iv$.
By \eqref{iti1}, one can find a point $u\in F\iv(\by)$, such that
\begin{gather}\label{iti4}
\tau\|x-u\|<\varphi(\|y-\by\|).
\end{gather}
By the convexity of $F$, the Fr\'echet normal cone to its graph coincides with the normal cone in the sense of convex analysis, and consequently it holds
\begin{gather*}
\langle x^*,u-x\rangle\le\langle v^*,\by-y\rangle +\xi_\varphi(y)\rho\|y-\by\|=-(1-\xi_\varphi(y)\rho)\|y-\by\|.
\end{gather*}
Combining this with \eqref{iti2}, \eqref{iti3}, and \eqref{iti4}, we have
\begin{align*}
\varphi'(\|y-\by\|)\|x^*\|\|u-x\| &\ge-\varphi'(\|y-\by\|)\langle x^*,u-x\rangle
\ge(\varphi'(\|y-\by\|)-\rho)\|y-\by\|
\\
&>\ga_2\varphi(\|y-\by\|) -(\ga_2-\ga_1)\varphi(\|y-\by\|)
=\ga_1\varphi(\|y-\by\|) >\ga_1\tau\|u-x\|.
\end{align*}
Hence,
\begin{align*}
\varphi'(\|y-\by\|)\|x^*\|>\ga_1\tau,
\end{align*}
and it follows from definitions \eqref{phi-sss} and \eqref{6srs} that $\overline{|\sd{F}|}_\varphi(\bar{x},\by)\ge\ga_1\tau$.
Passing to the limit in the last inequality as $\ga_1\to\vartheta[\varphi]$ and $\tau\to\sr_\varphi[F](\bar{x},\by)$, we arrive at the claimed inequality.
\qed\end{proof}

The next two corollaries summarize quantitative and qualitative criteria of metric $\varphi$-subregu\-larity.

\begin{corollary}[Quantitative criteria]\label{6C1.1}
Let $\ga>0$.
Consider the following conditions:
\renewcommand {\theenumi} {\alph{enumi}}
\begin{enumerate}
\item
$F$ is metrically $\varphi$-subregular at $(\bx,\by)$ with some $\tau>0$;
\item
$\overline{|\nabla{F}|}{}_\varphi^{\diamond}(\bar{x},\by)>\ga$,\\ i.e., for some $\rho>0$ and any $(x,y)\in\gph F$ with $x\notin F\iv(\by)$, $d(x,\bx)<\rho$, and $d(y,\by)<\rho$, it holds $|\nabla{F}|_{\varphi,\rho}^{\diamond}(x,y)>\ga$, and consequently there is a $(u,v)\in\gph F$, such that
\begin{gather*}
\varphi(d(y,\by))-\varphi(d(v,\by))>\ga d_\rho((u,v),(x,y));
\end{gather*}
\item
$\ds\liminf_{\substack{x\to\bx\\x\notin F\iv(\by),\,y\in F(x)}} \frac{\varphi(d(y,\by))}{d(x,\bx)}>\ga$;
\item
$\overline{|\nabla{F}|}{}_\varphi(\bar{x},\by)>\ga$,\\ i.e., for some $\rho>0$ and any $(x,y)\in\gph F$ with $x\notin F\iv(\by)$, $d(x,\bx)<\rho$, and $d(y,\by)<\rho$, it holds $\varphi'(d(y,\by))|\nabla{F}|_{\rho}(x,y)>\ga$, and consequently, for any $\eps>0$, there is a $(u,v)\in\gph F\cap B_\eps(x,y)$, such that
\begin{gather}\label{6phiree}
\varphi'(d(y,\by))(d(y,\by)-d(v,\by))>\ga d_\rho((u,v),(x,y));
\end{gather}
\item
$\overline{|\nabla{F}|}{}^+_\varphi(\bar{x},\by)>\ga$,\\ i.e., for some $\rho>0$ and any $(x,y)\in X\times Y$ with $x\notin F\iv(\by)$, $d(x,\bx)<\rho$, $d(y,\by)<\rho$, and $\varphi(d(y,\by))/d(x,\bx)\le\ga$, it holds $\varphi'(d(y,\by))|\nabla{F}|_{\rho}(x,y)>\ga$ and consequently, for any $\eps>0$, there is a $(u,v)\in\gph F\cap B_\eps(x,y)$, such that \eqref{6phiree} holds true;
\sloppy
\item
$X$ and $Y$ are normed spaces and
$\overline{|\sd{F}|}{}_\varphi^a(\bar{x},\by)>\ga$,\\ i.e.,
for some $\rho>0$ and any $(x,y)\in\gph F$ with $x\notin F\iv(\by)$, $\|x-\bx\|<\rho$, and $\|y-\by\|<\rho$, it holds $\varphi'(\|y-\by\|)|\sd{F}|^a_{\xi_\varphi(y)\rho}(x,y)>\ga$, and consequently there exists an $\eps>0$, such that
\begin{gather}\label{6f}
\varphi'(\|y-\by\|)\|x^*\|>\ga\;\;\mbox{for all } x^*\in D^*F(x,y)(J(B_\eps(y-\by))+\xi_\varphi(y)\rho\B^*);
\end{gather}
\item
$X$ and $Y$ are normed spaces and
$\overline{|\sd{F}|}{}_\varphi^{a+}(\bar{x},\by)>\ga$,\\ i.e., for some $\rho>0$ and any $(x,y)\in X\times Y$ with $x\notin F\iv(\by)$, $\|x-\bx\|<\rho$, $\|y-\by\|<\rho$, and $\varphi(\|y-\by\|)/\|x-\bx\|\le\ga$, it holds $\varphi'(\|y-\by\|)|\sd{F}|_{\xi_\varphi(y)\rho}^a(x,y)>\ga$ and consequently, there exists an $\eps>0$, such that
\eqref{6f} holds true;
\item
$X$ and $Y$ are normed spaces and
$\overline{|\sd{F}|}{}_\varphi(\bar{x},\by)>\ga$,\\ i.e.,
for some $\rho>0$ and any $(x,y)\in\gph F$ with $x\notin F\iv(\by)$, $\|x-\bx\|<\rho$, and $\|y-\by\|<\rho$, it holds $\varphi'(\|y-\by\|)|\sd{F}|_{\xi_\varphi(y)\rho}(x,y)>\ga$, and consequently
\begin{gather}\label{6f2}
\varphi'(\|y-\by\|)\|x^*\|>\ga\quad\mbox{for all } x^*\in D^*F(x,y)(J(y-\by)+\xi_\varphi(y)\rho\B^*);
\end{gather}
\sloppy
\item
$X$ and $Y$ are normed spaces and
$\overline{|\sd{F}|}{}^+_\varphi(\bar{x},\by)>\ga$,\\ i.e., for some $\rho>0$ and any $(x,y)\in X\times Y$ with $x\notin F\iv(\by)$, $\|x-\bx\|<\rho$, $\|y-\by\|<\rho$, and $\varphi(\|y-\by\|)/\|x-\bx\|\le\ga$, it holds $\varphi'(\|y-\by\|)|\sd{F}|_{\xi_\varphi(y)\rho}(x,y)>\ga$ and consequently,
\eqref{6f2} holds true;
\item
$X$ and $Y$ are finite dimensional normed spaces and
\begin{gather*}
\|x^*\|>\ga\quad\mbox{for all } x^*\in \overline{D}{}^{*>}_\varphi F(\bx,\by) (\Sp^*_{Y^*}).
\end{gather*}
\end{enumerate}
\renewcommand {\theenumi} {\roman{enumi}}
The following implications hold true:
\begin{enumerate}
\item
{\rm (c) \folgt (e)},
{\rm (d) \folgt (e)},
{\rm (e) \folgt (b)},
{\rm (f) \folgt (g) \folgt (i)}, {\rm (f)~\folgt (h)~\folgt (i)};
\item
if $\ga<\tau$, then {\rm (a) \folgt (b)};
\item
if $\tau\le\ga$, $X$ and $Y$ are complete, and $\gph F$ is locally closed near $(\bar{x},\by)$, then
{\rm (b) \folgt (a)}.
\cnta
\end{enumerate}
Suppose $X$ and $Y$ are normed spaces.
\begin{enumerate}
\cntb
\item
if $F$ is convex, and $\ga<\vartheta[\varphi]\tau$, then {\rm (a) \folgt (h)};
\item
{\rm (f)~\folgt (d)} and {\rm (g)~\folgt (e)},
provided that $X$ and $Y$ are Asplund and $\gph F$ is locally closed near $(\bx,\by)$;
\item
{\rm (h)~\iff (d)} and {\rm (i)~\iff (e)},\\
provided that
$Y$ is Fr\'echet smooth and one of the following conditions is satisfied:
\begin{enumerate}
\item
$X$ is Asplund and $\gph F$ is locally closed near $(\bar{x},\by)$;
\item
$F$ is convex;
\end{enumerate}
\item
{\rm (b)~\iff (d)~\iff (e)~\iff (h)~\iff (i)},
provided that
$F$ and $\varphi$ are convex;
\item
if $X$ and $Y$ are finite dimensional normed spaces, then {\rm (f)~\iff (h)~\iff (j)}.
%and  {\rm (g)~\iff (i)}.
\end{enumerate}
\end{corollary}

The conclusions of Corollary~\ref{6C1.1} are illustrated in
Fig.~\ref{fig.7}.

\begin{figure}[!htb]
$$\xymatrix{
&&&&{\rm (c)}\ar[d]
\\
&&{\rm (d)}\ar[rr]
\ar @{} [drr] |{\substack{X,Y\,{\rm Asplund}\\\gph F\, {\rm closed}}}
&&{\rm (e)}\ar[rr]
\ar@/_/@{-->}[ll]_{F,\varphi\,{\rm convex}}
\ar@/^1pc/@{-->}[dd]
&&{\rm (b)} \ar@/_/@{-->}[rr]_{\substack{\tau\le\ga\\X,Y\, {\rm complete}\\\gph F\, {\rm closed}}}
\ar@/_/@{-->}[ll]_{F,\varphi\,{\rm convex}}
&&{\rm (a)} \ar@/_/@{-->}[ll]_{\ga<\tau} \ar@/^5pc/@{-->}[lllllldd]|{\substack{X,Y\,{\rm normed}\\F\,{\rm convex}\\\ga<\vartheta[\varphi]\tau}}
\\
{\rm (j)}\ar@{-->}[rr]_(.65){\substack{\dim X<\infty\\\dim Y<\infty}}
\ar@{-->}[rrd]
&&{\rm (f)}\ar[rr]
\ar@{-->}[u]
\ar[d]
\ar@{-->}[ll]
&&{\rm (g)}
\ar@{-->}[u]
\ar[d]
\\
&&{\rm (h)}\ar[rr]\ar@{-->}[llu]\ar@/^/@{-->}[u]
&&{\rm (i)}
\ar@/_1pc/@{-->}[uu]_{\substack{[Y\,{\rm smooth}\smallskip\\(X\, {\rm Asplund}\\\gph F\, {\rm closed})\\{\rm or}\\ F\,{\rm convex}]\\{\rm or}\smallskip\\ F,\varphi\,{\rm convex}}}
\ar@/^/@{-->}[ll]^(.35){F,\varphi\,{\rm convex}}
}$$
\caption{Corollary~\ref{6C1.1} \label{fig.7}}
\end{figure}

\begin{corollary}[Qualitative criteria]\label{C8}
Suppose $X$ and $Y$ are complete metric spaces and $\gph F$ is locally closed near $(\bar{x},\by)$.
Then,
$F$ is metrically $\varphi$-subregular at $(\bx,\by)$, provided that one of the following conditions holds true:
\renewcommand {\theenumi} {\alph{enumi}}
\begin{enumerate}
\item
$\overline{|\nabla{F}|}{}^{\diamond}_\varphi(\bar{x},\by)>0$; \item
$\ds\liminf_{\substack{x\to\bx\\x\notin F\iv(\by),\,y\in F(x)}} \frac{\varphi(d(y,\by))}{d(x,\bx)}>0$;
\item
$\overline{|\nabla{F}|}{}_\varphi(\bar{x},\by)>0$
or equivalently,
$\ds\lim_{\rho\downarrow0}
\inf_{\substack{d(x,\bx)<\rho,\,d(y,\by)<\rho\\
(x,y)\in\gph F,\,x\notin F\iv(\by)}}\,
\varphi'(d(y,\by))|\nabla{F}|_{\rho}(x,y)>0;$
\item
$\overline{|\nabla{F}|}{}^+_\varphi(\bar{x},\by)>0$,
or equivalently,
$\ds\lim_{\rho\downarrow0}
\inf_{\substack{d(x,\bx)<\rho,\,\frac{\varphi(d(y,\by))}{d(x,\bx)}<\rho\\
(x,y)\in\gph F,\,x\notin F\iv(\by)}}\,
\varphi'(d(y,\by))|\nabla{F}|_{\rho}(x,y)>0.$
\cnta
\end{enumerate}
If $X$ and $Y$ are Asplund spaces, then the following conditions are also sufficient:
\begin{enumerate}
\cntb
\item
$\overline{|\sd{F}|}{}^a_\varphi(\bar{x},\by)>0$, or equivalently,
$\ds\lim_{\rho\downarrow0}
\inf_{\substack{\|x-\bx\|<\rho,\,\|y-\by\|<\rho\\
(x,y)\in\gph F,\,x\notin F\iv(\by)}}\,
\varphi'(\|y-\by\|)|\sd{F}|_{\xi_\varphi(y)\rho}^a(x,y)>0;$
\item
$\overline{|\sd{F}|}{}^{a+}_\varphi(\bar{x},\by)>0$, or equivalently,
$\ds\lim_{\rho\downarrow0}
\inf_{\substack{\|x-\bx\|<\rho,\,\frac{\varphi(\|y-\by\|)}{\|x-\bx\|}<\rho\\
(x,y)\in\gph F,\,x\notin F\iv(\by)}}\,
\varphi'(\|y-\by\|)|\sd{F}|_{\xi_\varphi(y)\rho}^a(x,y)>0.$
\cnta
\end{enumerate}
If $X$ and $Y$ are Banach spaces, then the next two conditions:
\begin{enumerate}
\cntb
\item
$\overline{|\sd{F}|}{}_\varphi(\bar{x},\by)>0$, or equivalently,
$\ds\lim_{\rho\downarrow0}
\inf_{\substack{\|x-\bx\|<\rho,\,\|y-\by\|<\rho\\
(x,y)\in\gph F,\,x\notin F\iv(\by)}}\,
\varphi'(\|y-\by\|)|\sd{F}|_{\xi_\varphi(y)\rho}(x,y)>0,$
\item
$\overline{|\sd{F}|}{}^+_\varphi(\bar{x},\by)>0$, or equivalently,
$\ds\lim_{\rho\downarrow0}
\inf_{\substack{\|x-\bx\|<\rho,\,\frac{\varphi(\|y-\by\|)}{\|x-\bx\|}<\rho\\
(x,y)\in\gph F,\,x\notin F\iv(\by)}}\,
\varphi'(\|y-\by\|)|\sd{F}|_{\xi_\varphi(y)\rho}(x,y)>0,$
\cnta
\end{enumerate}
are sufficient, provided that one of the following conditions is satisfied:
\begin{itemize}
\item
$X$ is Asplund and $Y$ is Fr\'echet smooth,
\item
$F$ is convex and either $Y$ is Fr\'echet smooth or $\varphi$ is convex.
\end{itemize}
If $X$ and $Y$ are finite dimensional normed spaces, then the following condition is also sufficient:
\begin{enumerate}
\cntb
\item
$0\notin \overline{D}{}^{*>}_\varphi F(\bx,\by) (\Sp^*_{Y^*})$.
\end{enumerate}
\renewcommand {\theenumi} {\roman{enumi}}
Moreover,
\begin{enumerate}
\item
condition {\rm (a)} is also necessary for the metric $\varphi$-subregularity of $F$ at $(\bx,\by)$;
\item
{\rm (b) \folgt (d)},
{\rm (c) \folgt (d)},
{\rm (d) \folgt (a)},
{\rm (e) \folgt (f) \folgt (h)},
{\rm (e) \folgt (g) \folgt (h)}.
\cnta
\end{enumerate}
Suppose $X$ and $Y$ are Banach spaces.
\begin{enumerate}
\cntb
\item
If $X$ and $Y$ are Asplund, then {\rm (e) \folgt (c)} and {\rm (f) \folgt (d)};
\item
if $Y$ is Fr\'echet smooth and either $X$ is Asplund or $F$ is convex near $(\bar{x},\by)$, then {\rm (e) \iff (c)} and {\rm (f) \iff (d)};
\item
if $F$ is convex near $(\bx,\by)$ and $\vartheta[\varphi]>0$, then
condition {\rm (g)} is also necessary for the metric $\varphi$-sub\-regularity of $F$ at $(\bx,\by)$;
\item
if $F$ is convex near $(\bx,\by)$ and $\varphi$ is convex near $0$, then
{\rm (a) \iff (c) \iff (d) \iff \rm (g) \iff (h)};
\item
if $X$ and $Y$ are finite dimensional normed spaces, then
{\rm (e) \iff (g) \iff (i)}.
\end{enumerate}
\end{corollary}

The conclusions of Corollary~\ref{C8} are illustrated in
Fig.~\ref{fig.8}.

\begin{figure}[!htb]
$$\xymatrix@C=1cm{
&&*+[F]{\sr_\varphi[F](\bx,\by)>0} \ar[d]
\\
&*+[F]{\ds\liminf_{\substack{x\to\bx\\x\notin F\iv(\by)\\ y\in F(x)}} \frac{\varphi(d(y,\by))}{d(x,\bx)}>0}\ar[dr]
&*+[F]{\overline{|\nabla{F}|}{}^{\diamond}_\varphi(\bar{x},\by)>0}
\ar[u]
\ar@/^/@{-->}[d]^{F,\varphi\,{\rm convex}}
\ar
%@/^9pc/
@{-->}`r /30pt[d] `^rd/4pt[dd] `_l[l] `[l]|{\substack{X,Y\,{\rm Banach}\\F\,{\rm convex}\\\vartheta[\varphi]>0}} [lddd]
\\
&*+[F]{\overline{|\nabla{F}|}_\varphi(\bar{x},\by)>0}
\ar[r]
\ar@/^/@{-->}[d]
\ar@{} [dr] |{\substack{X\,{\rm Banach},\,Y\,{\rm  smooth}\smallskip\\X\,{\rm Asplund}\,{\rm or}\,F\,{\rm convex}}}
&*+[F]{\overline{|\nabla{F}|}{}^{+}_\varphi(\bar{x},\by)>0}
\ar[u]
\ar@/_/@{-->}[d]
\ar@/_/@{-->}[l]_{F,\varphi\,{\rm convex}}
\\
*+[F]{0\notin \overline{D}{}^{*>}_\varphi F(\bx,\by) (\Sp^*_{Y^*})}
\ar@{-->}[r]
\ar@{-->} [dr]^(.65){\substack{\dim X<\infty\\\dim Y<\infty}} &*+[F]{\overline{|\sd{F}|}{}^a_\varphi(\bar{x},\by)>0}
\ar[r]
\ar@{-->}[l]
\ar@/^/@{-->}[u]^{X,Y\,{\rm Asplund}}
\ar[d]
&*+[F]{\overline{|\sd{F}|}{}^{a+}_\varphi(\bar{x},\by)>0}
\ar@/_/@{-->}[u]_{X,Y\,{\rm Asplund}}
\ar[d]
\\
&*+[F]{\overline{|\sd{F}|}_\varphi(\bar{x},\by)>0}
\ar[r]
\ar@{-->}@/^/[u]
\ar@{-->}[ul]
&*+[F]{\overline{|\sd{F}|}{}^{+}_\varphi(\bar{x},\by)>0}
\ar@/_4.5pc/@{-->}[uu]|{\substack{X,Y\,{\rm Banach}\\F,\varphi\,{\rm convex}}}
\ar@/^/@{-->}[l]^{F,\varphi\,{\rm convex}}
}$$
\caption{Corollary~\ref{C8}}\label{fig.8}
\end{figure}

The next example illustrates the computation of the constants involved in the definition and characterizations of metric $\varphi$-subregularity.

\begin{example}
Consider a mapping $F:\R\to\R$ given by
$
F(x):=
1-\cos x.
$
One has $(0,0)\in\gph F$, $F(x)>0$ for all $x\ne0$ near 0 and $\lim_{x\to0}F(x)/x=0$.
Hence, $F$ is not metrically subregular at $(0,0)$.
Define
$$
\varphi(t):=
\begin{cases}
\arccos(1-t) & \text{if } 0\le t<\frac{1}{2},\\
\frac{\pi}{3}+\frac{2t-1}{\sqrt{3}} & \text{if } t\ge\frac{1}{2}.
\end{cases}
$$
Then $\varphi(0)=0$ and $\varphi$ is continuously differentiable on $\R_+$ with
$$
\varphi'(t):=
\begin{cases}
+\infty & \text{if } t=0,\\
\frac{1}{\sqrt{t(2-t)}} & \text{if } 0<t<\frac{1}{2},\\
\frac{2}{\sqrt{3}} & \text{if } t\ge\frac{1}{2}.
\end{cases}
$$
(Thanks to Remark~\ref{R5.1}, it is sufficient to define $\varphi$ near 0 only.)
The modulus of metric $\varphi$-subregularity \eqref{6CMR-g} can be easily computed:
\begin{gather*}
\sr_\varphi[F](0,0)=
\liminf_{\substack{x\to0\\x\notin F\iv(0)}} \frac{\varphi(d(0,F(x)))}{d(x,F^{-1}(0))}=
\lim_{\substack{x\to0}} \frac{\varphi(1-\cos x)} {|x|}=
\lim_{\substack{x\to0}} \frac{\arccos(\cos x)} {|x|}=1.
\end{gather*}
Hence, $F$ is metrically $\varphi$-subregular at $(0,0)$ with constant 1.

This result can also be deduced from Theorem~\ref{6T1}(ii).
For that, one needs to compute the uniform strict $\varphi$-slope \eqref{6uss}.
Let $x\ne0$, $|x|<\pi/3$, $y=1-\cos x$, and $\rho\in(0,1)$.
Then the nonlocal $(\varphi,\rho)$-slope \eqref{6nls} of $F$ at $(x,y)$ takes the following form:
\begin{align*}
|\nabla{F}|_{\varphi,\rho}^{\diamond}(x,y)&=
\sup_{\substack{(u,v)\ne(x,y)\\(u,v)\in\gph F}}
\frac{\varphi(|y|)-\varphi(|v|)}{d_\rho((u,v),(x,y))}=
\sup_{u\ne x}
\frac{|x|-\varphi(1-\cos u)} {\max\{|u-x|,\rho|\cos u-\cos x|\}}=\sup_{u\ne x}
\frac{|x|-\varphi(1-\cos u)} {|u-x|}.
\end{align*}
If $1/2<\cos u\le1$, then
\begin{gather*}
\frac{|x|-\varphi(1-\cos u)} {|u-x|}=\frac{|x|-|u|} {|u-x|}\le1,
\end{gather*}
and the equality holds when $u=0$.
If $\cos u\le1/2$, then $|u|\ge\pi/3$ and
\begin{gather*}
\frac{|x|-\varphi(1-\cos u)} {|u-x|}=\frac{|x|-\frac{\pi}{3}-\frac{1-2\cos u} {\sqrt{3}}} {|u-x|}\le\frac{|x|-\frac{\pi}{3}} {|u-x|}\le1.
\end{gather*}
Hence, $|\nabla{F}|_{\varphi,\rho}^{\diamond}(x,y)=1$,
and consequently
$\overline{|\nabla{F}|}{}^{\diamond}_\varphi(0,0)=1$.

Metric $\varphi$-subregularity of $F$ can also be established from the estimates in Corollaries~\ref{6C1.1} and \ref{C8} after computing any of the local strict $\varphi$-slopes \eqref{phi-ss}, \eqref{mphi-ss}, \eqref{phi-asss} and \eqref{phi-masss} or the limiting outer $\varphi$-coderivative \eqref{D*21}.
The first two constants, in their turn, depend on the local $\rho$-slopes \eqref{srho} and \eqref{6asrs}.
Observe that the last two constants do not depend on $\varphi$.
For instance, the $\rho$-slope \eqref{srho} and the strict $\varphi$-slope \eqref{phi-ss} can be computed similarly to the above.
Let $x\ne0$, $|x|<\pi/3$, $y=1-\cos x$, and $\rho\in(0,1)$.
Then
\begin{gather*}
|\nabla{F}|_{\rho}(x,y)=
\limsup_{\substack{(u,v)\to(x,y),\,
(u,v)\ne(x,y)\\
(u,v)\in\gph F}}
\frac{[|y|-|v|]_+} {d_\rho((u,v),(x,y))}=
\limsup_{\substack{u\to x,\,
u\ne x}}
\frac{[\cos u-\cos x]_+} {|u-x|}=|\sin x|,
\\
\overline{|\nabla{F}|}{}_\varphi(0,0)=
\liminf_{\substack{x\to0,\,x\ne0}}\,
\varphi'(1-\cos x)\;|\sin x|=
\liminf_{\substack{x\to0,\,x\ne0}}\,
\frac{|\sin x|}{\sqrt{(1-\cos x)(1+\cos x)}}=1.
\end{gather*}
\qedtr
\end{example}

\subsection{H\"older Metric Subregularity}\label{S5}

Let a real number $q\in(0,1]$ be given.

A set-valued mapping $F:X\rightrightarrows Y$ between metric spaces is called H\"older metrically subregular of order $q$ at $(\bx,\by)\in\gph F$ with constant $\tau>0$ iff there exists a neighbourhood $U$ of $\bx$, such that
\begin{equation}\label{HMR}
\tau d(x,F^{-1}(\by))\le (d(\by,F(x)))^q\quad \mbox{for all } x\in U.
\end{equation}

This property is a special case of the metric $\varphi$-subregularity property when
\begin{gather}\label{gq}
\varphi(t)=t^q,\quad t\in\R_+.
\end{gather}
It is easy to check, that function $\varphi$ defined by \eqref{gq} is continuously differentiable (with possibly infinite $\varphi'(0)$ understood as the right-hand derivative) and satisfies conditions ($\Phi1$) and ($\Phi2$).
In particular,
\begin{gather*}
\varphi'(t)=qt^{q-1},\quad t\in\R_+\setminus\{0\}
\quad\text{and}\quad
\varphi'(0)=
\begin{cases}
1 & \text{if } q=1,\\
+\infty & \text{if } 0<q<1.
\end{cases}
\end{gather*}

The representations and estimates of the previous section are applicable and lead to a series of criteria of H\"older metric subregularity; cf. \cite{Kru15.2}.

\section{Conclusions}
This article demonstrates how nonlinear metric subregularity properties of set-valued mappings between general metric or Banach spaces can be treated in the framework of the theory of (linear) error bounds for extended real-valued functions of two variables and provides a comprehensive collection of quantitative and qualitative regularity criteria with the relationships between the criteria identified and illustrated.
Several kinds of primal and subdifferential slopes of set-valued mappings are used in the criteria.

\section*{Acknowledgements}
The research was supported by the Australian Research Council, project DP110102011.

The author wishes to thank two of the three anonymous referees for the careful reading of the manuscript and many constructive comments and suggestions.

\section*{Conflict of Interest}
The author declares that he has no conflict of interest.
%\section*{Acknowledgements}

%The authors wish to acknowledge the very thorough refereeing of the article and express their gratitude to the two anonymous referees for valuable and helpful comments and suggestions.

\bibliographystyle{spmpsci_unsrt}
\bibliography{buch-kr,kruger,kr-tmp}
\end{document}

%% file: macro.tex
\lccode`\-=`\-
\newcommand{\al}{\alpha}

\newcommand{\ga}{\gamma}

\newcommand{\de}{\delta}
\newcommand{\eps}{\varepsilon}
\newcommand{\bx}{\bar x}
\newcommand{\by}{\bar y}

\newcommand{\iv}{^{-1} }
\newcommand{\Real}{\mathbb R}
\newcommand {\R} {\mathbb R}
\newcommand {\N} {\mathbb N}

\newcommand {\B} {\mathbb B}
\newcommand {\Sp} {\mathbb S}
\newcommand {\gph} {{\rm gph}\,}%Graph
\newcommand {\dom} {{\rm dom}\,}

\newcommand {\Er} {{\rm Er}\,}

\newcommand {\sd} {\partial}
\renewcommand{\iff}{$ \Leftrightarrow\ $}%iff
\newcommand{\folgt}{$ \Rightarrow\ $}

\newcommand{\ds}{\displaystyle}

\def\lsc{lower semicontinuous}

\def\RHS{right-hand side}
\def\SVM{set-valued mapping}

\newcommand{\norm}[1]{\left\Vert#1\right\Vert}

\newcommand{\set}[1]{\left\{#1\right\}}

\newcounter{mycount}
\def\cnta{\setcounter{mycount}{\value{enumi}}}
\def\cntb{\setcounter{enumi}{\value{mycount}}}